\newif\ifLTEX
\LTEXtrue

\documentclass[a4paper]{amsart}

\usepackage{mathtools,amssymb}
\usepackage[abbrev]{amsrefs}
\usepackage{xparse}

\ifLTEX
\usepackage{graphicx}
\else
\usepackage[dvipdfmx]{graphicx}
\fi
 
\usepackage{hyperref}
\usepackage{subcaption}
\usepackage{enumitem}
\usepackage{amsthm}
\usepackage{pdfpages}
\usepackage[all]{xy} 
\usepackage{float}
\usepackage{marginfix}
\newtheorem{thm}{Theorem}[section]
\newtheorem{lem}[thm]{Lemma}

\newtheorem{prop}[thm]{Proposition}

\theoremstyle{definition}

\newcommand{\R}{\mathbb{R}}

\newcommand{\Z}{\mathbb{Z}}

\newcommand{\M}{\mathrm{Mod}_{0,2n+2}}
\newcommand{\Mb}{\mathrm{Mod}_{0,2n+1}^1}
\newcommand{\Mp}{\mathrm{Mod}_{0,2n+2;\ast }}
\newcommand{\LM}{\mathrm{LMod}_{2n+2}}
\newcommand{\LMb}{\mathrm{LMod}_{2n+1}^1}
\newcommand{\LMp}{\mathrm{LMod}_{2n+2,\ast }}
\newcommand{\PM}{\mathrm{PMod}_{0,2n+2}}

\newcommand{\Mg}{\mathrm{Mod}_{g}}
\newcommand{\Mgb}{\mathrm{Mod}_{g}^1}

\newcommand{\SM}{\mathrm{SMod}_{g;k}}
\newcommand{\SMb}{\mathrm{SMod}_{g;k}^1}
\newcommand{\SMp}{\mathrm{SMod}_{g,\ast ;k}}
\newcommand{\B}{\mathcal{B}}

\numberwithin{equation}{section}

\allowdisplaybreaks
\title[Small generating set for balanced superelliptic mapping class group]{The balanced superelliptic mapping class groups are generated by three elements}
\author[G.~Omori]{Genki Omori}
\address{
(Genki Omori)
Department of Mathematics, Faculty of Science and Technology, Tokyo University of Science, 2641 Yamazaki, Noda-shi, Chiba, 278-8510 Japan
}
\email{omori\_genki@ma.noda.tus.ac.jp}

\if0
\author[]{}
\address{
()
}
\email{}
\fi


\date{\today}

\begin{document}
\maketitle
\begin{abstract}
The balanced superelliptic mapping class group is the normalizer of the transformation group of the balanced superelliptic covering in the mapping class group of the total surface.  
We prove that the balanced superelliptic mapping class groups with either one marked point, one boundary component, or no marked points and boundary are generated by three elements. 
To prove this, we also show that its liftable mapping class groups are also generated by three elements. 
These generating sets are minimal except for several no marked points and boundary cases.  
\end{abstract}

\section{Introduction}

%\subsection{Background}

Let $\Sigma _{g}$ be a connected closed oriented surface of genus $g\geq 0$. 
For subsets $A$ and $B$ of $\Sigma _g$ $(g\geq 0)$, the {\it mapping class group} $\mathrm{Mod}(\Sigma _{g}; A, B)$ of the tuple $(\Sigma _{g}, A, B)$ is the group of isotopy classes of orientation-preserving self-homeomorphisms on $\Sigma _{g}$ which preserve $A$ setwise and $B$ pointwise. 
When $A$ is a set of $n$ points and $B$ is a disjoint union of $b$ disks such that $A$ and $B$ have no intersections, we denote the mapping class group by $\mathrm{Mod}_{g,n}^b$. 
We denote simply $\mathrm{Mod}_{g}^b=\mathrm{Mod}_{g,0}^b$, $\mathrm{Mod}_{g,n}=\mathrm{Mod}_{g,n}^0$, and $\mathrm{Mod}_{g}=\mathrm{Mod}_{g,0}^0$. 

Dehn~\cite{Dehn} proved that $\mathrm{Mod}_{g,n}^b$ is generated by Dehn twists. 
After that small finite generating sets for $\mathrm{Mod}_{g}$ by Dehn twists were given by Lickorish~\cite{Lickorish} and Humphries~\cite{Humphries}, and finite presentations for $\mathrm{Mod}_{g,n}^b$ were given by Gervais~\cite{Gervais} for $g\geq 1$ and $n=0$ and by Labruere-Paris~\cite{Labruere-Paris} for $g\geq 1$ and $b\geq 1$. 
Generators of these presentations for $n=0, 1$ are consist of Dehn twists. 
In~\cite{Humphries}, Humphries also proved that $\mathrm{Mod}_{g}$ does not generated by at most $2g$ Dehn twists and his generating set is minimal in generating sets for $\Mg $ which consist of Dehn twists. 

As a study of more smaller generating set for the mapping class group, Wajnryb~\cite{Wajnryb} proved that $\mathrm{Mod}_{g}^b$ is generated by two elements for $g\geq 1$ and $b\in \{0,1\}$. 
Since $\mathrm{Mod}_{g}^b$ is not a cyclic group, this Wajnryb's generating set is minimal. 
After that, Korkmaz~\cite{Korkmaz} proved that $\mathrm{Mod}_{g}^b$ is generated by two elements whose one element is a Dehn twist for $g\geq 1$ and $b\in \{0,1\}$. 
Monden~\cite{Monden} showed that $\mathrm{Mod}_{g,n}$ is generated by two elements for $g\geq 3$ and $n\geq 0$, and this generating set is also minimal. 

For integers $n\geq 1$ and $k\geq 2$ with $g=n(k-1)$, the \textit{balanced superelliptic covering map} $p=p_{g,k}\colon \Sigma _g\to \Sigma _0$ is a branched covering map with $2n+2$ branch points $p_1,\ p_2,\ \dots ,\ p_{2n+2}\in \Sigma _{0}$, their unique preimages $\widetilde{p}_1,\ \widetilde{p}_2,\ \dots ,\ \widetilde{p}_{2n+2}\in \Sigma _{g}$, and the covering transformation group generated by the \textit{balanced superelliptic rotation} $\zeta =\zeta _{g,k}$ of order $k$ (precisely defined in Section~\ref{section_bscov} and see Figure~\ref{fig_bs_periodic_map}). 
When $k=2$, $\zeta =\zeta _{g,2}$ coincides with the hyperelliptic involution, and for $k\geq 3$, the balanced superelliptic covering space was introduced by Ghaswala and Winarski~\cite{Ghaswala-Winarski2}. 

Let $D$ be a disk on $\Sigma _0-\{ p_1,\ p_2,\ \dots ,\ p_{2n+1}\}$ whose interior includes the point $p_{2n+2}$ and $\widetilde{D}$ the preimage of $D$ with respect to $p$.  
Then, denote by $\Sigma _g^1$ for $g\geq 1$ the complement of the interior of $\widetilde{D}$ in $\Sigma _g$. 
%We also regard $\zeta =\zeta _{g,k}$ as a self-homeomorphism on $\Sigma _g^1$. 
For $g=n(k-1)\geq 1$, an orientation-preserving self-homeomorphism $\varphi $ on $\Sigma _{g}$ (resp. on $\Sigma _{g}^1$ whose restriction to $\partial \Sigma _{g}^1$ is the identity map) is \textit{symmetric} for $\zeta =\zeta _{g,k}$ if $\varphi \left< \zeta \right> \varphi ^{-1}=\left< \zeta \right> $ (resp. $\varphi \left< \zeta |_{\Sigma _g^1} \right> \varphi ^{-1}=\left< \zeta |_{\Sigma _g^1} \right> $). 
We abuse notation and denote simply $\zeta |_{\Sigma _g^1}=\zeta $. 
%For $g=n(k-1)\geq 1$, an orientation-preserving self-homeomorphism $\varphi $ on $\Sigma _{g}$ or on $\Sigma _{g}^1$ whose restriction to $\partial \Sigma _{g}^1$ is the identity map is \textit{symmetric} for $\zeta $ if $\varphi \left< \zeta \right> \varphi ^{-1}=\left< \zeta \right> $. 
We regard $\mathrm{Mod}_{g,1}$ and $\Mgb $ as the groups $\mathrm{Mod}(\Sigma _g; \{ \widetilde{p}_{2n+2}\} , \emptyset )$ and $\mathrm{Mod}(\Sigma _g; \emptyset , \widetilde{D})$, respectively. 
The \textit{balanced superelliptic mapping class group} (or \textit{symmetric mapping class group}) $\SM $ (resp. $\SMp$ and $\SMb$) is the subgroup of $\Mg $ (resp. $\mathrm{Mod}_{g,1}$ and $\Mgb$) which consists of elements represented by symmetric homeomorphisms. 
Birman and Hilden~\cite{Birman-Hilden3} showed that $\SM $ (resp. $\SMp$ and $\SMb$) coincides with the group of symmetric isotopy classes of symmetric homeomorphisms on $\Sigma _g$ (resp. $(\Sigma _g, \widetilde{p}_{2n+2})$ and $\Sigma _g^1$). 

We denote by $\Sigma _0^1$ the complement of the interior of $D$ in $\Sigma _0$. 
An orientation-preserving self-homeomorphism $\varphi $ on $\Sigma _{0}$ (resp. on $\Sigma _{0}^1$ whose restriction to $\partial \Sigma _{0}^1$ is the identity map) is \textit{liftable} with respect to %the balanced superelliptic covering map 
$p=p_{g,k}$ if there exists an orientation-preserving self-homeomorphism $\widetilde{\varphi }$ on $\Sigma _{g}$ (resp. on $\Sigma _{g}^1$ whose restriction to $\partial \Sigma _{g}^1$ is the identity map) such that $p\circ \widetilde{\varphi }=\varphi \circ p$ (resp. $p|_{\Sigma _g^1}\circ \widetilde{\varphi }=\varphi \circ p|_{\Sigma _g^1}$), namely, the following diagrams commute: 
\[
\xymatrix{
\Sigma _g \ar[r]^{\widetilde{\varphi }} \ar[d]_p &  \Sigma _{g} \ar[d]^p & \Sigma _g^1 \ar[r]^{\widetilde{\varphi }}\ar[d]_{p|_{\Sigma _g^1}}  &  \Sigma _{g}^1\ar[d]^{p|_{\Sigma _g^1}} \\
\Sigma _{0}  \ar[r]_{\varphi } &\Sigma _{0}, \ar@{}[lu]|{\circlearrowright} & \Sigma _{0}^1  \ar[r]_{\varphi } &\Sigma _{0}^1. \ar@{}[lu]|{\circlearrowright}
}
\] 
Put $\B =\{ p_1,\ p_2,\ \dots ,\ p_{2n+2}\}$. 
We regard $\M $ and $\Mb $ as the groups $\mathrm{Mod}(\Sigma _0; \B , \emptyset )$ and $\mathrm{Mod}(\Sigma _0; \B -\{ p_{2n+2}\} , D)$, respectively, and denote by $\Mp $ the group $\mathrm{Mod}(\Sigma _0; \B -\{ p_{2n+2}\} , \{ p_{2n+2}\} )$. 
The \textit{liftable mapping class group} $\mathrm{LMod}_{2n+2;k}$ (resp. $\mathrm{LMod}_{2n+2,\ast ;k}$ and $\mathrm{LMod}_{2n+1;k}^1$) is the subgroup of $\M $ (resp. $\Mp $ and $\Mb $) which consists of elements represented by liftable homomorphisms.
Birman and Hilden~\cite{Birman-Hilden2} proved that $\mathrm{LMod}_{2n+1;k}^1$ is isomorphic to $\SMb $ and $\mathrm{LMod}_{2n+2;k}$ (resp. $\mathrm{LMod}_{2n+2,\ast ;k}$) is isomorphic to the quotient of $\SM $ (resp. $\SMp$) by $\left< \zeta \right> $.  

When $k=2$, $\zeta =\zeta _{g,2}$ is the hyperelliptic involution and corresponding symmetric mapping class groups are called the hyperelliptic mapping class groups.  
In this case, $\mathrm{LMod}_{2n+2;2}$ (resp. $\mathrm{LMod}_{2n+2,\ast ;2}$ and $\mathrm{LMod}_{2n+1;2}^1$) is equal to $\M $ (resp. $\Mp $ and $\Mb $) and a finite presentation for $\mathrm{SMod}_{g;2}$ was given by Birman and Hilden~\cite{Birman-Hilden1}. 
Stukow~\cite{Stukow} gave a minimal generating set for $\mathrm{SMod}_{g;2}$ by two torsion elements. 

When $k\geq 3$, Ghaswala and Winarski~\cite{Ghaswala-Winarski2} proved that $\mathrm{LMod}_{2n+2;k}$ is a proper subgroup of $\M $. 
They also gave a finite presentation for $\mathrm{LMod}_{2n+2;k}$ in~\cite{Ghaswala-Winarski1}.
By Lemma~3.6 in~\cite{Ghaswala-Winarski1} and Lemmas~4.2 and 4.4 in~\cite{Hirose-Omori}, we show that the group $\mathrm{LMod}_{2n+2,\ast ;k}$ (resp. $\mathrm{LMod}_{2n+1;k}^1$) is also a proper subgroup of $\Mp $ (resp. $\Mb $) for $k\geq 3$ and also show that $\mathrm{LMod}_{2n+2;k}=\mathrm{LMod}_{2n+2;l}$, $\mathrm{LMod}_{2n+2,\ast ;k}=\mathrm{LMod}_{2n+2,\ast ;l}$, and $\mathrm{LMod}_{2n+1;k}^1=\mathrm{LMod}_{2n+1;l}^1$ for $l>k\geq 3$. 
Hence we omit ``$k$'' in the notation of the liftable mapping class groups for $k\geq 3$ (i.e. we express $\mathrm{LMod}_{2n+2;k}=\LM $, $\mathrm{LMod}_{2n+2,\ast ;k}=\LMp $, and $\mathrm{LMod}_{2n+1;k}^1=\LMb $ for $k\geq 3$). 
Hirose and the author~\cite{Hirose-Omori} gave another finite presentation for $\LM $ and finite presentations for the groups $\LMp$, $\LMb$, $\SM $, $\SMp $, and $\SMb $. 

The main theorems in this paper are as follows. 

\begin{thm}\label{thm_lmod}
$\LMp $ and $\LMb $ for $n\geq 2$ and $\LM $ for $n\geq 1$ are generated by three elements. 
$\mathrm{LMod}_{4,\ast }$ and $\mathrm{LMod}_{3}^1$ are generated by two elements. 
\end{thm}

\begin{thm}\label{thm_smod}
Assume that $n\geq 1$ and $k\geq 3$ with $g=n(k-1)$. 
Then, $\SMp $ and $\SMb $ for $n\geq 2$ and $\SM $ for $n\geq 1$ are generated by three elements. 
$\mathrm{SMod}_{g,\ast ;k}$ and $\mathrm{SMod}_{g;k}^1$ for $n=1$ are generated by two elements. 
\end{thm}

Theorems~\ref{thm_lmod} and \ref{thm_smod} are proved in Sections~\ref{section_lmod} and \ref{section_smod}. 
Theorems~\ref{thm_lmod} and \ref{thm_smod} for $n=1$ is immediately obtained from Corollaries~5.4 and 6.13 in~\cite{Hirose-Omori}. 
The integral first homology group $H_1(G)$ of a group $G$ is isomorphic to the abelianization of $G$. 
By Theorem~1.1 in \cite{Ghaswala-Winarski1-1} and Theorems~1.1 and 1.2 in~\cite{Hirose-Omori}, the integral first homology groups of the liftable mapping class groups and the balanced superelliptic mapping class groups are as follows: 

\begin{thm}[Theorem~1.1 in \cite{Ghaswala-Winarski1-1} and Theorem~1.1 in~\cite{Hirose-Omori}]\label{thm_abel_lmod}
For $n\geq 1$,
\begin{enumerate}
\item $H_1(\LM )\cong \left\{ \begin{array}{ll}
 \Z \oplus \Z _2\oplus \Z _{2}&\text{if }  n \text{ is odd},   \\
 \Z \oplus \Z _{2}&\text{if }  n \text{ is even},
 \end{array} \right.$\\
\item $H_1(\LMp )\cong \left\{ \begin{array}{ll}
 \Z \oplus \Z _2&\text{if }  n=1,   \\
 \Z ^2\oplus \Z _{n}&\text{if }  n\geq 2 \text{ is even},\\
 \Z ^2\oplus \Z _{2n}&\text{if }  n\geq 3 \text{ is odd},
 \end{array} \right.$\\
\item $H_1(\LMb )\cong \left\{ \begin{array}{ll}
 \Z ^2&\text{if }  n=1,   \\
 \Z ^3&\text{if }  n\geq 2.
 \end{array} \right.$\\
\end{enumerate}
\end{thm}

\begin{thm}[Theorem~1.2 in~\cite{Hirose-Omori}]\label{thm_abel_smod}
For $n\geq 1$ and $k\geq 3$ with $g=n(k-1)$,
\begin{enumerate}
\item $H_1(\SM )\cong \left\{ \begin{array}{ll}
 \Z \oplus \Z _2\oplus \Z _2&\text{if }  n\geq 1\text{ is odd and }k\text{ is odd},   \\
 \Z \oplus \Z _2\oplus \Z _4&\text{if }  n\geq 1\text{ is odd and }k\text{ is even},   \\ 
 \Z \oplus \Z _2 &\text{if }  n\geq 2\text{ is even},
 \end{array} \right.$\\
\item $H_1(\SMp )\cong \left\{ \begin{array}{ll}
 \Z \oplus \Z _{2k}&\text{if }  n=1,   \\
 \Z ^2\oplus \Z _{kn}&\text{if }  n\geq 2 \text{ is even},\\
 \Z ^2\oplus \Z _{2kn}&\text{if }  n\geq 3 \text{ is odd},
 \end{array} \right.$\\
\item $H_1(\SMb )\cong \left\{ \begin{array}{ll}
 \Z ^2&\text{if }  n=1,   \\
 \Z ^3&\text{if }  n\geq 2.
 \end{array} \right.$\\
\end{enumerate}
\end{thm}

For a group $G$, the minimal number of generators for $H_1(G)$ gives a lower bound of the minimal number of generators for $G$.  
By Theorems~\ref{thm_abel_lmod} and \ref{thm_abel_smod}, we can show that the generating sets in Theorem~\ref{thm_lmod} or \ref{thm_smod} are minimal except for the cases that $\LM $ for even $n$ and $\SM $ for even $n$. 

%\subsection{Main result}\label{section_main-result}

\section{Preliminaries}\label{Preliminaries}

\subsection{The balanced superelliptic covering space}\label{section_bscov}

In this section, we review the definition of the balanced superelliptic coverring space from Section~2.1 in \cite{Hirose-Omori}. 
%Let $\Sigma _{g}$ be a connected closed oriented surface of genus $g\geq 0$. 
For integers $n\geq 1$ and $k\geq 2$ with $g=n(k-1)$, we describe the surface $\Sigma _{g}$ as follows. 
We take the unit 2-sphere $S^2=S(1)$ in $\R ^3$ and $n$ mutually disjoint parallel copies $S(2),\ S(3),\ \dots ,\ S(n+1)$ of $S(1)$ by translations along the x-axis such that 
\[
\max \bigl( S(i)\cap (\R \times \{ 0\}\times \{ 0\} )\bigr) <\min \bigl( S(i+1)\cap (\R \times \{ 0\}\times \{ 0\} )\bigr)
\]
for $1\leq i\leq n$ (see Figure~\ref{fig_bs_periodic_map}). 
Let $\zeta$ be the $(-\frac{2\pi }{k})$-rotation of $\R ^3$ on the $x$-axis. 
Then we remove $2k$ disjoint open disks in $S(i)$ for $2\leq i\leq n$ and $k$ disjoint open disks in $S(i)$ for $i\in \{ 1,\ n+1\}$ which are setwisely preserved by the action of $\zeta $, and connect $k$ boundary components of the punctured $S(i)$ and ones of the punctured $S(i+1)$ by $k$ annuli such that the union of the $k$ annuli is preserved by the action of $\zeta $ for each $1\leq i\leq n$ as in Figure~\ref{fig_bs_periodic_map}. 
Since the union of the punctured $S(1)\cup S(2)\cup \cdots \cup S(n+1)$ and the attached $n\times k$ annuli is homeomorphic to $\Sigma _{g=n(k-1)}$, 
%where $g=n(k-1)$, 
we regard this union as $\Sigma _g$. 

By the construction above, the action of $\zeta $ on $\R ^3$ induces the action on $\Sigma _g$, the quotient space $\Sigma _g/\left< \zeta \right>$ is homeomorphic to $\Sigma _0$, and the quotient map $p=p_{g,k}\colon \Sigma _g\to \Sigma _0$ is a branched covering map with $2n+2$ branch points in $\Sigma _0$. 
We call the branched covering map $p\colon \Sigma _g\to \Sigma _0$ the \textit{balanced superelliptic covering map}. 
Denote by $\widetilde{p}_1,\ \widetilde{p}_2,\ \dots ,\ \widetilde{p}_{2n+2}\in \Sigma _g$ the fixed points of $\zeta $ such that $\widetilde{p}_i<\widetilde{p}_{i+1}$ in $\R =\R \times \{ 0\} \times \{ 0\}$ for $1\leq i\leq 2n+1$, by $p_i\in \Sigma _0$ for $1\leq i\leq 2n+2$ the image of $\widetilde{p}_i$ by $p$ (i.e. $p_1,\ p_2,\ \dots ,\ p_{2n+2}\in \Sigma _0$ are branch points of $p$), and by $\B $ the set of the branch points $p_1,\ p_2,\ \dots ,\ p_{2n+2}$. 
Let $D$ be a 2-disk in $(\Sigma _0-\B )\cup \{ p_{2n+2}\}$ whose interior includes the point $p_{2n+2}$ and $\widetilde{D}$ the preimage of $D$ with respect to $p$. 
Denote by $\Sigma _g^1$ the complement of interior of $\widetilde{D}$ in $\Sigma _g$ and by $\Sigma _0^1$ the complement of interior of $D$ in $\Sigma _0$. 
Then we also call the restriction $p|_{\Sigma _g^1}\colon \Sigma _g^1\to \Sigma _0^1$ the balanced superelliptic covering map and we denote simply $p|_{\Sigma _g^1}=p$. 

\begin{figure}[h]
\includegraphics[scale=1.5]{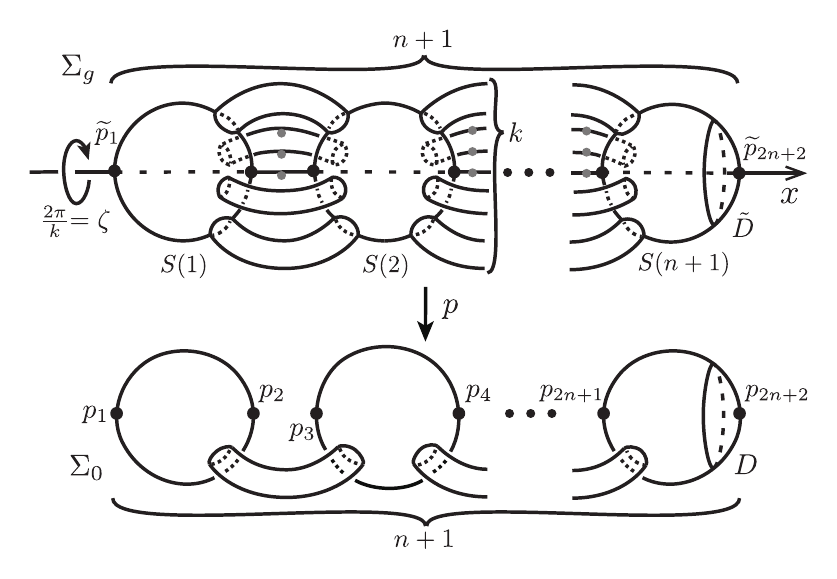}
\caption{The balanced superelliptic covering map $p=p_{g,k}\colon \Sigma _g\to \Sigma _0$.}\label{fig_bs_periodic_map}
\end{figure}

\subsection{Liftable elements for the balanced superelliptic covering map}\label{section_liftable-element}

In this section, we introduce some liftable homeomorphisms on $\Sigma _0$ for the balanced superelliptic covering map $p=p_{g,k}$ for $k\geq 3$ from Section~2.2 in \cite{Hirose-Omori}. 
Let $l_i$ $(1\leq i\leq 2n+1)$ be an oriented simple arc on $\Sigma _0$ whose endpoints are $p_i$ and $p_{i+1}$ as in Figure~\ref{fig_path_l}. 
Put $L=l_1\cup l_2\cup \cdots \cup l_{2n+1}$. %, $L^1=L\cap \Sigma _0^1$, and $\B ^1=\B -\{ p_{2n+2}\}$. 
The isotopy class of a homeomorphism $\varphi $ on $\Sigma _0$ (resp. $\Sigma _0^1$) relative to $\B $ (resp. $(\B -\{ p_{2n+2}\})\cup \partial \Sigma _0^1$) is determined by the isotopy class of the image of $L$ (resp. $L\cap \Sigma _0^1$) by $\varphi $ relative to  $\B $ (resp. $(\B -\{ p_{2n+2}\})\cup \partial \Sigma _0^1$). 
We identify $\Sigma _0$ with the surface on the lower side in Figure~\ref{fig_path_l} by some homeomorphism of $\Sigma _0$. 

\begin{figure}[h]
\includegraphics[scale=1.5]{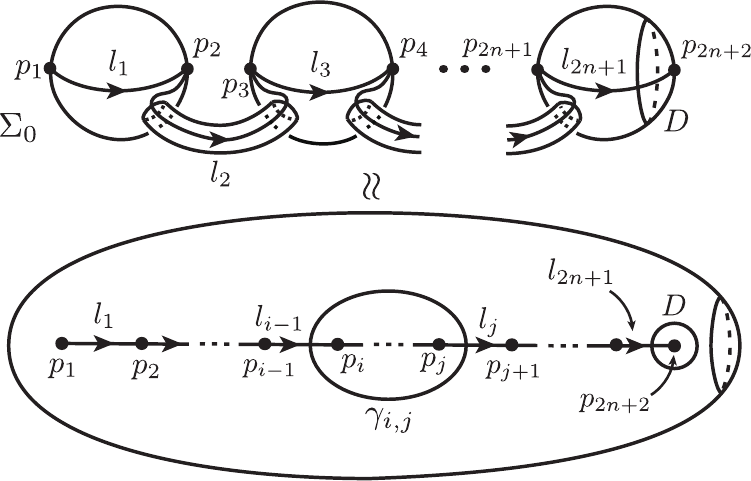}
\caption{A natural homeomorphism of $\Sigma _0$, arcs $l_1,\ l_2,\ \dots ,\ l_{2n+1}$, and a simple closed curve $\gamma _{i,j}$ on $\Sigma _0$ for $1\leq i<j\leq 2n+2$.}\label{fig_path_l}
\end{figure}

Let $\sigma _i$ for $1\leq i\leq 2n+1$ be a self-homeomorphism on $\Sigma _0$ which is described as the result of anticlockwise half-rotation of $l_i$ in the regular neighborhood of $l_i$ in $\Sigma _0$ as in Figure~\ref{fig_sigma_l}. 
$\sigma _i$ is called the \textit{half-twist} along $l_i$. 
As a well-known result, $\M $ is generated by $\sigma _1$, $\sigma _2,\ \dots $, $\sigma _{2n+1}$ (see for instance Section~9.1.4 in \cite{Farb-Margalit}). 
Since $\M $ is naturally acts on $\B $, we have the surjective homomorphism 
\[
\Psi \colon \M \to S_{2n+2}
\]
given by $\Psi (\sigma _i)=(i\ i+1)$, where $S_{2n+2}$ is the symmetric group of degree $2n+2$. 

\begin{figure}[h]
\includegraphics[scale=1.1]{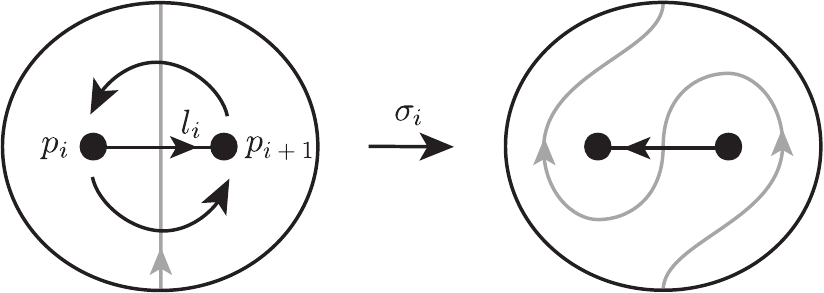}
\caption{The half-twist $\sigma _i$ for $1\leq i\leq 2n+1$.}\label{fig_sigma_l}
\end{figure}

Put $\B _o=\{ p_1,\ p_3,\ \dots ,\ p_{2n+1}\}$ and $\B _e=\{ p_2,\ p_4,\ \dots ,\ p_{2n+2}\}$. 
An element $\sigma $ in $S_{2n+2}$ is \textit{parity-preserving} if $\sigma (\B _o)=\B _o$, and is \textit{parity-reversing} if $\sigma (\B _o)=\B _e$. 
An element $f$ in $\M $ is \textit{parity-preserving} (resp. \textit{parity-reversing}) if $\Psi (f)$ is \textit{parity-preserving} (resp. \textit{parity-reversing}). 
Let $W_{2n+2}$ be the subgroup of $S_{2n+2}$ which consists of parity-preserving or parity-reversing elements, $S_{n+1}^o$ (resp. $S_{n+1}^e$) the subgroup of $S_{2n+2}$ which consists of elements whose restriction to $\B _e$ (resp. $\B _o$) is the identity map.  
Note that $S_{n+1}^o$ (resp. $S_{n+1}^e$) is a subgroup of $W_{2n+2}$, isomorphic to $S_{n+1}$, and generated by transpositions $(1\ 3)$, $(3\ 5),\ \dots $, $(2n-1\ 2n+1)$ (resp. $(2\ 4)$, $(4\ 6),\ \dots $, $(2n\ 2n+2)$). 
Then we have the following exact sequence:
\begin{eqnarray}\label{exact1}
1\longrightarrow S_{n+1}^o\times S_{n+1}^e\longrightarrow W_{2n+2}\stackrel{\pi }{\longrightarrow }\Z _2\longrightarrow 1, 
\end{eqnarray}
where the homomorphism $\pi \colon W_{2n+2}\to \Z _2$ is defined by $\pi (\sigma )=0$ if $\sigma $ is parity-preserving and $\pi (\sigma )=1$ if $\sigma $ is parity-reversing. 
Ghaswala and Winarski~\cite{Ghaswala-Winarski1} proved the following lemma. 
\begin{lem}[Lemma~3.6 in \cite{Ghaswala-Winarski1}]\label{lem_GW}
Let $\LM $ be the liftable mapping class group for the balanced superelliptic covering map $p_{g,k}$ for $n\geq 1$ and $k\geq3$ with $g=n(k-1)$. 
Then we have
\[
\LM =\Psi ^{-1}(W_{2n+2}).
\]
\end{lem}
Lemma~\ref{lem_GW} implies that a mapping class $f\in \M $ lifts with respect to $p_{g,k}$ if and only if $f$ is parity-preserving or parity-reversing (in particular, when $k\geq 3$, the liftability of a homeomorphism on $\Sigma _0$ or $\Sigma _0^1$ does not depend on $k$). 
The \textit{pure mapping class group} $\PM $ is the kernel of the homomorphism $\Psi \colon \M \to S_{2n+2}$. 
Since all elements in $\PM $ is parity-preserving, $\PM $ is a subgroup of $\LM $ and we have the following exact sequence:  
\begin{eqnarray}\label{exact2}
1\longrightarrow \PM \longrightarrow \LM \stackrel{\Psi }{\longrightarrow }W_{2n+2}\longrightarrow 1. 
\end{eqnarray}
For mapping classes $[\varphi ]$ and $[\psi ]$, $[\psi ][\varphi ]$ means $[\psi \circ \varphi ]$, i.e. $[\varphi ]$ apply first, and we often abuse notation and denote a homeomorphism and its isotopy class by the same symbol.  
We will introduce some explicit liftable elements as follows. 
\\

\if0
\paragraph{\emph{The Dehn twist $t_{i,i+1}$}}

For a simple closed curve $\gamma $ on $\Sigma _g$ $(g\geq 0)$, we denote by $t_\gamma $ the right-handed Dehn twist along $\gamma $. 
Let $\gamma _{i,i+1}$ for $1\leq i\leq 2n+1$ be a simple closed curve on $\Sigma _0-(\B \cup D)$ such that $\gamma _{i,i+1}$ surrounds the two points $p_i$ and $p_{i+1}$ as in Figure~\ref{fig_path_l}. 
Then we define $t_{i,i+1}=t_{\gamma _{i,i+1}}$ for $1\leq i\leq 2n+1$. 
Since the Dehn twist $t_{i,i+1}$ preserves $\B$ pointwise, i.e. $t_{i,j}$ lies in $\PM $, $t_{i,i+1}$ lifts with respect to $p$ by Lemma~\ref{lem_GW}. 
\fi

\paragraph{\emph{The Dehn twist $t_{i,j}$}}

For a simple closed curve $\gamma $ on $\Sigma _g$ $(g\geq 0)$, we denote by $t_\gamma $ the right-handed Dehn twist along $\gamma $. 
Let $\gamma _{i,j}$ for $1\leq i<j\leq 2n+2$ be a simple closed curve on $\Sigma _0-(\B \cup D)$ such that $\gamma _{i,j}$ surrounds the $j-i+1$ points $p_i$, $p_{i+1},\ \dots $,~$p_j$ as in Figure~\ref{fig_path_l}. 
We have $\gamma _{i,2n+2}=\gamma _{1,i-1}$ for $i\geq 3$ as an isotopy class of a curve. 
Then we define $t_{i,j}=t_{\gamma _{i,j}}$ for $1\leq i<j\leq 2n+2$. 
Note that $t_{1,2n+1}=t_{1,2n+2}=t_{2,2n+2}=1$ and $t_{i,2n+2}=t_{1,i-1}$ for $i\geq 3$ in $\M $. 
Since the Dehn twist $t_{i,j}$ preserves $\B$ pointwise, i.e. $t_{i,j}$ lies in $\PM $, $t_{i,j}$ lifts with respect to $p$ by Lemma~\ref{lem_GW}. \\

\paragraph{\emph{The half-rotation $h_i$}}

Let $\mathcal{N}$ be a regular neighborhood of $l_i\cup l_{i+1}$ in $(\Sigma _0-\B )\cup \{ p_{i},\ p_{i+1},\ p_{i+2}\}$ for $1\leq i\leq 2n$. 
$\mathcal{N}$ is homeomorphic to a 2-disk. 
Then we denote by $h_i$ the self-homeomorphism on $\Sigma _0$ which is described as the result of anticlockwise half-rotation of $l_i\cup l_{i+1}$ in $\mathcal{N}$ as in Figure~\ref{fig_h_i}. 
Note that $h_i=\sigma _i\sigma _{i+1}\sigma _i$ for $1\leq i\leq 2n$. 
Since $\Psi (h_i)=(i\ i+2)$ for $1\leq i\leq 2n$, the mapping class $h_i$ is parity-preserving. 
Thus, by Lemma~\ref{lem_GW}, $h_i$ lifts with respect to $p$. 

\begin{figure}[h]
\includegraphics[scale=0.95]{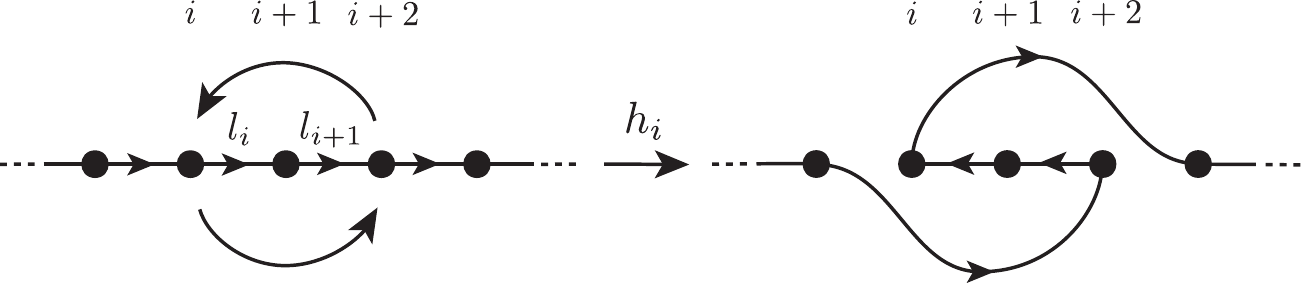}
\caption{The mapping class $h_i$ for $1\leq i\leq 2n$ on $\Sigma _0$.}\label{fig_h_i}
\end{figure}

\paragraph{\emph{The parity-reversing element $r_1$}}

Let $r_1$ be the self-homeomorphism on $\Sigma _0$ which is described as the result of the $\frac{2\pi }{2n+2}$-rotation on $\Sigma _0$ as in Figure~\ref{fig_reverse_r_1}. 
We see that $r(l_{i})=l_{i+1}$ for $1\leq i\leq 2n$, the order of $r_1$ is $2n+2$, and $\Psi (r_1)=(1\ 2)(2\ 3)\cdots (2n+1\ 2n+2)$. 
Thus $r_1$ is parity-reversing and lifts with respect to $p$ by Lemma~\ref{lem_GW}. \\

\begin{figure}[h]
\includegraphics[scale=1.3]{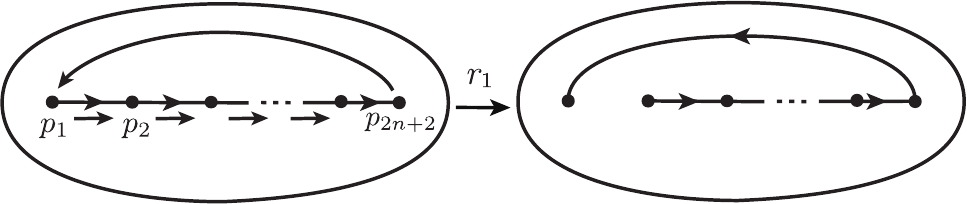}
\caption{The mapping class $r_1$ on $\Sigma _0$.}\label{fig_reverse_r_1}
\end{figure}

\paragraph{\emph{The parity-reversing element $r$}}

Let $r$ be the self-homeomorphism on $\Sigma _0$ which is described as the result of the $\pi $-rotation of $\Sigma _0$ on the axis as in Figure~\ref{fig_r}. 
We see that $r(l_{i})=l_{2n+2-i}^{-1}$ for $1\leq i\leq 2n+1$ and $\Psi (r)=(1\ 2n+2)(2\ 2n+1)\cdots (n+1\ n+2)$. 
Thus $r$ is parity-reversing and lifts with respect to $p$ by Lemma~\ref{lem_GW}. \\

\begin{figure}[h]
\includegraphics[scale=1.3]{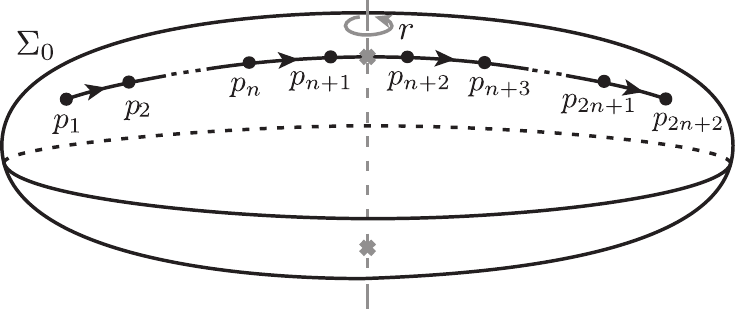}
\caption{The mapping class $r$ on $\Sigma _0$.}\label{fig_r}
\end{figure}

We have the following lemma. 

\begin{lem}\label{lem-r_1-product}
For $n\geq 1$, put 
\[
F=\left\{
		\begin{array}{ll}
		h_1^{-1} \quad \text{for }n=1,\\
		(h_1^{-1}h_2^{-1}\cdots h_{2n-2}^{-1})(h_1^{-1}h_2^{-1}\cdots h_{2n-4}^{-1})\cdots (h_1^{-1}h_2^{-1})h_{2n-1}^{-1}\cdots h_3^{-1}h_1^{-1}\\
		t_{2n,2n+1}^{n-1}\cdots t_{6,7}^2t_{4,5}\quad \text{for }n\geq 2.\\
		\end{array}
		\right.
\]
Then, the relation $r_1=rF$ holds in $\LM $. 
\end{lem} 

\begin{proof}
Since $r$ is an involution and the isotopy class of a homeomorphism $\varphi $ on $\Sigma _0$ relative to $\B $ is determined by the isotopy class of the image $\varphi (L)$ relative to $\B $, we will prove that $F^{-1}rr_1(l_i)=l_i$ for $1\leq i\leq 2n+1$ relative to $\B $. 
The images $r_1(L)$ and $rr_1(L)$ are described as on the second and third steps in Figure~\ref{fig_proof-lem-r_1}. 
When $n=1$, we clearly show that $F^{-1}rr_1(l_i)=h_1rr_1(l_i)=l_i$ for $1\leq i\leq 2n+1$ relative to $\B $. 

Assume that $n\geq 2$. 
Put $F_i=t_{2i+1,2i+2}^{-i}h_{2i}\cdots h_2h_1$ for $1\leq i\leq n-1$. 
By applying $F_1\cdots F_{n-2}F_{n-1}=(t_{3,4}^{-1}h_2h_1)\cdots (t_{2n-3,2n-2}^{-(n-2)}h_{2n-4}\cdots h_2h_1)(t_{2n-1,2n}^{-(n-1)}h_{2n-2}\cdots h_2h_1)$ to $rr_1(L)$ inductively, we have the image $F_1\cdots F_{n-2}F_{n-1}rr_1(L)$ as on the sixth step in Figure~\ref{fig_proof-lem-r_1}.
Finally, by applying $h_1h_3\cdots h_{2n-1}$ to this image, we show that $h_1h_3\cdots h_{2n-1}F_1\cdots F_{n-2}F_{n-1}rr_1(l_i)=l_i$ for $1\leq i\leq 2n+1$ as on the bottom in Figure~\ref{fig_proof-lem-r_1}. 
Since $t_{i,i+1}$ for $1\leq i\leq 2n+1$ commutes with $h_j$ for $j\not \in \{ i-2,\ i-1,\ i,\ i+1\}$ and the relation $h_it_{i,i+1}=t_{i+1,i+2}h_i$ for $1\leq i\leq 2n$ holds in $\LM $, we have
\begin{eqnarray*}
&&h_1h_3\cdots h_{2n-1}F_1\cdots F_{n-2}F_{n-1}\\
&=&h_1h_3\cdots h_{2n-1}\\
&&\cdot (\underset{\leftarrow }{\underline{t_{3,4}^{-1}}}h_2h_1)\cdots (\underset{\leftarrow }{\underline{t_{2n-3,2n-2}^{-(n-2)}}}h_{2n-4}\cdots h_2h_1)(\underset{\leftarrow }{\underline{t_{2n-1,2n}^{-(n-1)}}}h_{2n-2}\cdots h_2h_1)\\
&\overset{\text{COMM}}{\underset{}{=}}&h_1\underline{h_3t_{3,4}^{-1}}\cdots \underline{h_{2n-3}t_{2n-3,2n-2}^{-(n-2)}}\ \underline{h_{2n-1}t_{2n-1,2n}^{-(n-1)}}\\
&&\cdot (h_2h_1)\cdots (h_{2n-4}\cdots h_2h_1)(h_{2n-2}\cdots h_2h_1)\\
&\overset{}{\underset{}{=}}&h_1\underset{\leftarrow }{\underline{t_{4,5}^{-1}}}h_3\cdots h_{2n-5}\underset{\leftarrow }{\underline{t_{2n-2,2n-1}^{-(n-2)}}}h_{2n-3}\underset{\leftarrow }{\underline{t_{2n,2n+1}^{-(n-1)}}}h_{2n-1}\\
&&\cdot (h_2h_1)\cdots (h_{2n-4}\cdots h_2h_1)(h_{2n-2}\cdots h_2h_1)\\
&\overset{\text{COMM}}{\underset{}{=}}&t_{4,5}^{-1}\cdots t_{2n-2,2n-1}^{-(n-2)}t_{2n,2n+1}^{-(n-1)}h_1h_3\cdots h_{2n-1}\\
&&\cdot (h_2h_1)\cdots (h_{2n-4}\cdots h_2h_1)(h_{2n-2}\cdots h_2h_1)\\
&=&F^{-1},
\end{eqnarray*}
where $A\stackrel{\text{COMM}}{=}B$ means that $B$ is obtained from $A$ by commutative relations and ``$A_1\underset{\rightarrow }{\underline{A}}A_2$'' (resp. ``$A_1\underset{\leftarrow }{\underline{A}}A_2$'') means that deforming the word $A_1AA_2$ by moving $A$ right (resp. left). 
Thus we have $F^{-1}rr_1(l_i)=h_1h_3\cdots h_{2n-1}F_1\cdots F_{n-2}F_{n-1}rr_1(l_i)=l_i$ for $1\leq i\leq 2n+1$ and 
\begin{eqnarray*}
F^{-1}rr_1=1\Longleftrightarrow r_1=rF
\end{eqnarray*}
in $\LM $. 
Therefore, we have completed the proof of Lemma~\ref{lem-r_1-product}. 
\end{proof}

\begin{figure}[h]
\includegraphics[scale=1.4]{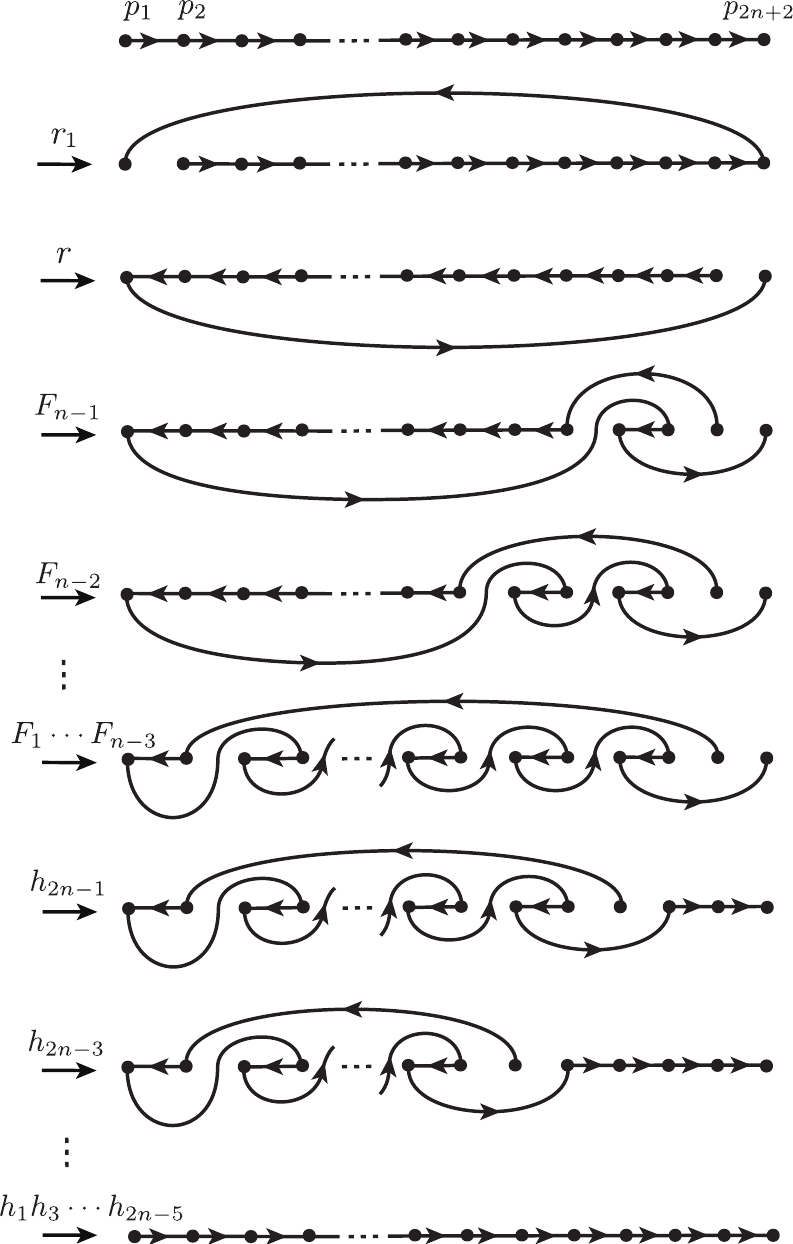}
\caption{The proof of $h_1h_3\cdots h_{2n-1}F_1\cdots F_{n-2}F_{n-1}rr_1(l_i)=l_i$ for $1\leq i\leq 2n+1$.}\label{fig_proof-lem-r_1}
\end{figure}

We can regard a mapping class in $\M $ which has a representative fixing $D$ pointwise as an element in $\Mb $. 
Hence we also regard the mapping classes $t_{i,i+1}$ for $1\leq i\leq 2n+1$ and $h_i$ for $1\leq i\leq 2n-1$ as elements in $\Mb $.

\subsection{Explicit lifts of liftable elements for the balanced superelliptic covering map}\label{section_lifts}

In this section, from Section~6.1 in \cite{Hirose-Omori}, we review explicit lifts of liftable elements for the balanced superelliptic covering map $p=p_{g,k}\colon \Sigma _g\to \Sigma _0$ which are introduced in Section~\ref{section_liftable-element}. 
Throughout this section, we assume that $g=n(k-1)$ for $n\geq 1$ and $k\geq 3$.  

Let $\widetilde{l}_i^l$ for $1\leq i\leq 2n+1$ and $1\leq l\leq k$ be a lift of $l_i$ with respect to $p$ such that $\zeta (\widetilde{l}_i^l)=\widetilde{l}_i^{l+1}$ for  $1\leq l\leq k-1$ and $\zeta (\widetilde{l}_i^k)=\widetilde{l}_i^{1}$. 
We consider a homeomorphism of $\Sigma _g$ as in Figure~\ref{fig_isotopy_surface_3-handles} and identify $\Sigma _g$ with the surface as on the lower side in Figure~\ref{fig_isotopy_surface_3-handles}. 

\begin{figure}[h]
\includegraphics[scale=0.9]{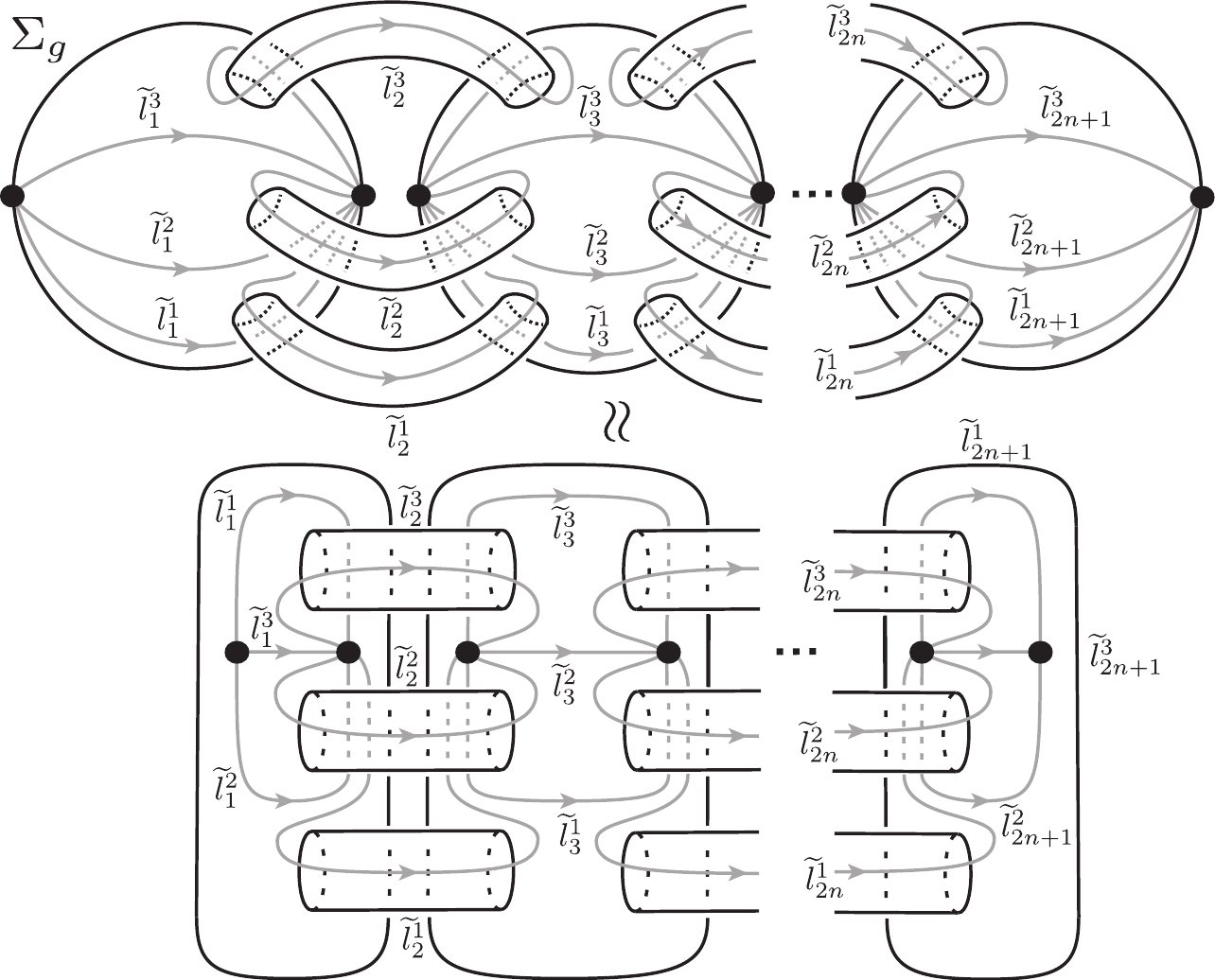}
\caption{A natural homeomorphism of $\Sigma _g$ when $k=3$.}\label{fig_isotopy_surface_3-handles}
\end{figure}

A simple closed curve $\gamma $ on $\Sigma _0-\B $ \textit{lifts} with respect to $p$ if there exists a simple closed curve $\widetilde{\gamma }$ on $\Sigma _{g}-p^{-1}(\B )$ such that the restriction $p|_{\widetilde{\gamma }}\colon \widetilde{\gamma } \to \gamma $ is bijective. 
By Lemma~3.3 in \cite{Ghaswala-Winarski1}, a simple closed curve $\gamma $ on $\Sigma _0-\B $ lifts with respect to $p$ if and only if the algebraic intersection number of $\gamma $ and $l_1\cup l_3\cup \cdots \cup l_{2n+1}$ is zero mod $k$. 
Hence the simple closed curve $\gamma _{i,i+1}$ on $\Sigma _0$ for $1\leq i\leq 2n+1$ (see Figure~\ref{fig_path_l}) lifts with respect to $p$. 
By the information of intersections of $\gamma _{i,i+1}$ and each $l_{j}$ $(1\leq j\leq 2n+1)$, we can show that a lift of $\gamma _{i,i+1}$ for $1\leq i\leq 2n+1$ with respect to $p$ by the simple closed curve $\gamma _{i}^l$ on $\Sigma _g$ for $1\leq l\leq k$ as in Figure~\ref{fig_scc_c_il} (precisely, see before Lemma~6.4 in \cite{Hirose-Omori}). 

\begin{figure}[h]
\includegraphics[scale=0.90]{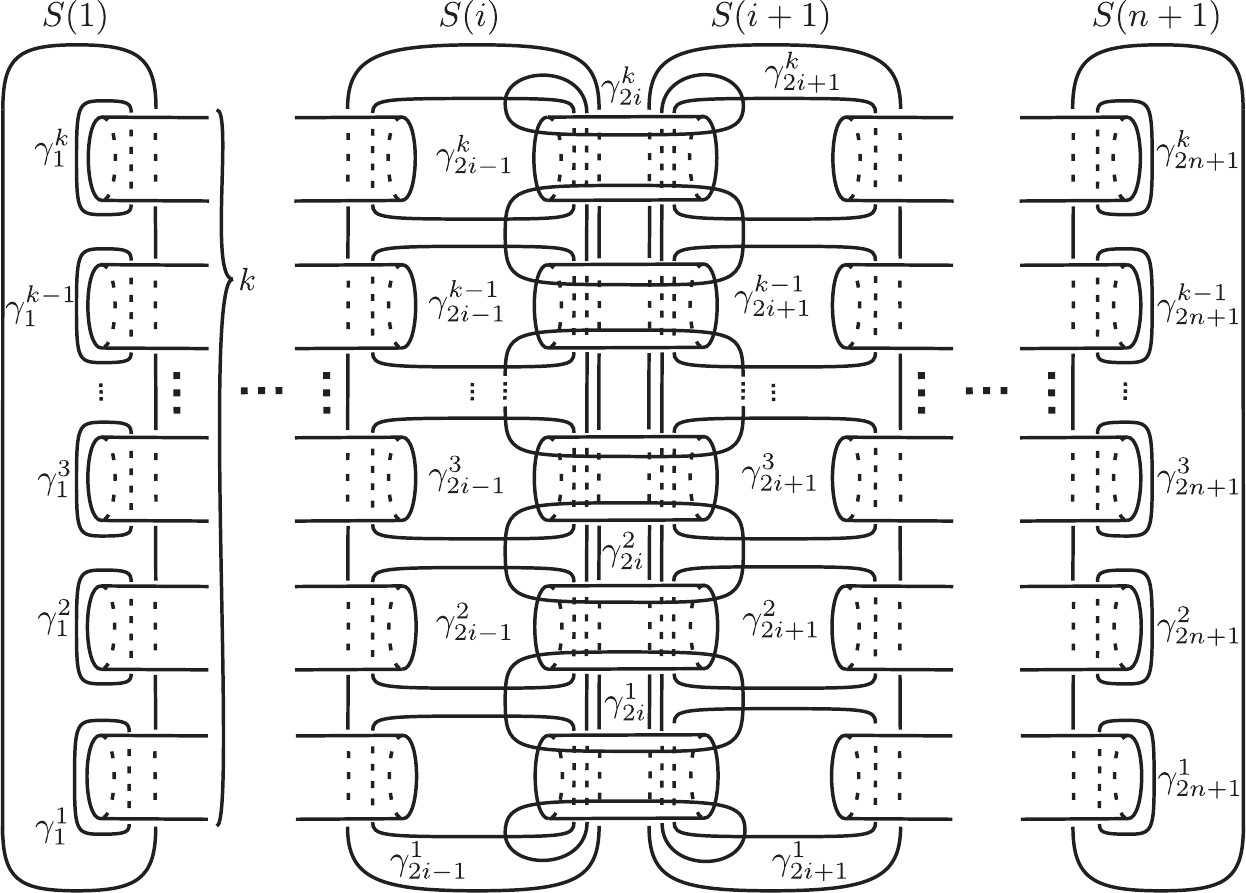}
\caption{Simple closed curves $\gamma _i^l$ on $\Sigma _g$ for $1\leq i\leq 2n+1$ and $1\leq l\leq k$.}\label{fig_scc_c_il}
\end{figure}

Put $\widetilde{t}_{i,i+1}=t_{\gamma _i^1}t_{\gamma _i^2}\cdots t_{\gamma _i^{k}}$ for $1\leq i\leq 2n+1$. 
By Lemma~6.4 in \cite{Hirose-Omori}, we have the following lemma. 

\begin{lem}\label{lift-t_{i,i+1}}
For $1\leq i\leq 2n+1$, $\widetilde{t}_{i,i+1}$ is a lift of $t_{i,i+1}$ with respect to $p$. 
\end{lem}

Put $\widetilde{h}_i=t_{\gamma _i^1}t_{\gamma _{i+1}^1}t_{\gamma _i^2}t_{\gamma _{i+1}^2}\cdots t_{\gamma _i^{k-1}}t_{\gamma _{i+1}^{k-1}}t_{\gamma _i^{k}}$ for odd $1\leq i\leq 2n-1$ and $\widetilde{h}_i=t_{\gamma _i^{k}}t_{\gamma _{i+1}^{k}}t_{\gamma _i^{k-1}}t_{\gamma _{i+1}^{k-1}}\cdots t_{\gamma _i^{2}}t_{\gamma _{i+1}^{2}}t_{\gamma _i^{1}}$ for even $2\leq i\leq 2n$.
Then, by Lemma~6.6 in \cite{Hirose-Omori}, we have the following lemma. 

\begin{lem}\label{lift-h_i}
For $1\leq i\leq 2n$, $\widetilde{h}_{i}$ is a lift of $h_{i}$ with respect to $p$. 
\end{lem}

\if0
The inclusion $\widetilde{\iota }\colon \Sigma _g^1\hookrightarrow \Sigma _g=\Sigma _g^1\cup \widetilde{D}$ induces the surjective homomorphism $\widetilde{\iota }_\ast \colon \Mgb \to \mathrm{Mod}_{g,1}$ by extending homeomorphisms by the identity map on $\widetilde{D}$. 
The homomorphism $\widetilde{\iota }_\ast$ is called the \textit{capping homomorphism}. 
By an argument in the proof of Proposition~4.5 in \cite{Hirose-Omori}, the fact that $t_{\partial D}=t_{1,2n+1}$ in $\LMb $, and using the capping homomorphism $\widetilde{\iota }_\ast$, we have the following lemma.  

\begin{lem}\label{lift-zeta}
$\zeta $ is a lift of $t_{1,2n+1}$ with respect to $p$. 
\end{lem}
\fi

Let $\widetilde{r}$ be a homeomorphism on $\Sigma _g$ which satisfies $\widetilde{r}(\widetilde{l}_{2i-1}^l)=(\widetilde{l}_{2n-2i+3}^{k-l+2})^{-1}$ for $1\leq i\leq n+1$ and $2\leq l\leq k$ and $\widetilde{r}(\widetilde{l}_{2i}^{l})=(\widetilde{l}_{2n-2i+2}^{k-l+1})^{-1}$ for $1\leq i\leq n$ and $1\leq l\leq k$. 
Remark that $\widetilde{r}$ is uniquely determined up to isotopy and is described as the result of the $\pi $-rotation of $\Sigma _g$ as in Figure~\ref{fig_lift_r}. 
The rotation axis of $\widetilde{r}$ intersects with $\Sigma _g$ at two points in $S(\frac{n+1}{2})$ for odd $n$, at two points in the handle intersecting with $\widetilde{l}_{n+1}^{\frac{k+1}{2}}$ for even $n$ and odd $k$, and at no points for even $n$ and $k$. 
In actuality, the arc $\widetilde{l}_{n+1}^{\frac{k+1}{2}}$ in the case of even $n$ and odd $k$ must intersect with the rotation axis of $\widetilde{r}$, however, we describe $\widetilde{l}_{n+1}^{\frac{k+1}{2}}$ in Figure~\ref{fig_lift_r} as a parallel arc of $\widetilde{l}_{n+1}^{\frac{k+1}{2}}$ to see well. 
Since $r(l_i)=l_{2n-i+2}^{-1}$ for $1\leq i\leq 2n+1$, $\widetilde{r}$ is a lift of $r$ with respect to $p$. 

\begin{figure}[h]
\includegraphics[scale=1.1]{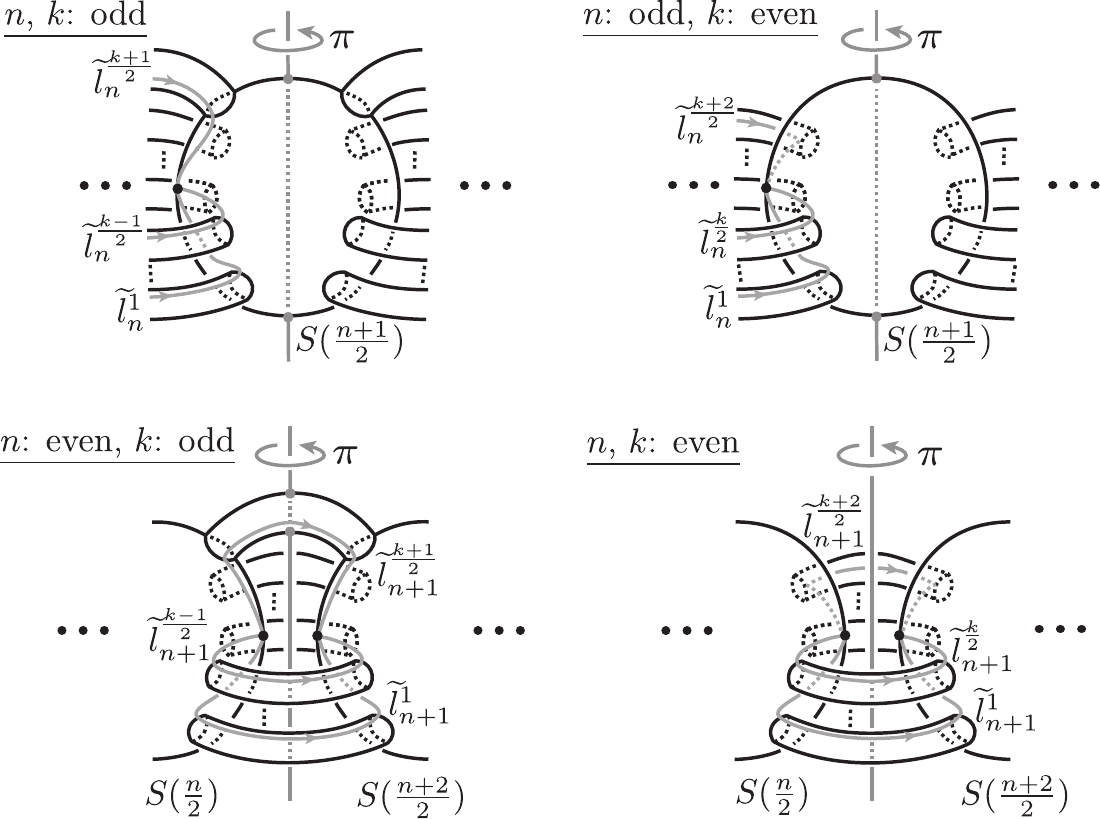}
\caption{A lift $\widetilde{r}$ of $r$ with respect to $p$.}\label{fig_lift_r}
\end{figure}

Put $\widetilde{r}_1=\widetilde{r}(\widetilde{h}_1^{-1}\widetilde{h}_2^{-1}\cdots \widetilde{h}_{2n-2}^{-1})(\widetilde{h}_1^{-1}\widetilde{h}_2^{-1}\cdots \widetilde{h}_{2n-4}^{-1})\cdots (\widetilde{h}_1^{-1}\widetilde{h}_2^{-1})\widetilde{h}_{2n-1}^{-1}\cdots \widetilde{h}_3^{-1}\widetilde{h}_1^{-1}\cdot \\ \widetilde{t}_{2n,2n+1}^{n-1}\cdots \widetilde{t}_{6,7}^2\widetilde{t}_{4,5}$ for $n\geq 2$ and $\widetilde{r}_1=\widetilde{r}\widetilde{h}_1^{-1}$ for $n=1$. 
By Lemmas~\ref{lem-r_1-product}, \ref{lift-t_{i,i+1}}, and \ref{lift-h_i} and the fact that $\widetilde{r}$ is a lift of $r$ with respect to $p$, we have the following lemma. 

\begin{lem}\label{lift-r_1}
$\widetilde{r}_{1}$ is a lift of $r_{1}$ with respect to $p$. 
\end{lem}

\section{Small generating sets for liftable mapping class groups}\label{section_lmod}

Theorem~\ref{thm_lmod} follows from the next proposition. 

\begin{prop}\label{prop_LM}
\begin{enumerate}
\item For $n\geq 1$, $\LM $ is generated by $h_1$, $t_{1,2}$, and $r_1$. 
\item For $n\geq 2$, $\LMp $ and $\LMb $ are generated by $h_1$, $h_2$, and $h_{2n-1}\cdots h_2h_1t_{1,2}$. 
\item $\mathrm{LMod}_{4,\ast }$ and $\mathrm{LMod}_{3}^1$ are generated by $h_1$ and $t_{1,2}$. 
\end{enumerate}
\end{prop} 

We remark that Proposition~\ref{prop_LM} for $\mathrm{LMod}_{4,\ast }$ or $\mathrm{LMod}_{3}^1$ is immediately obtained from Corollary~5.4 in~\cite{Hirose-Omori} (see also Lemma~\ref{lem_LMp}). 
By Proposition~\ref{prop_LM}, we can see that $\LM $ is generated by two infinite order elements and one torsion element with order $2n+2$. 
In this section, we will prove Theorem~\ref{thm_lmod} by showing Proposition~\ref{prop_LM}. 

\subsection{Proof of the main theorem for $\boldsymbol{\LM }$}\label{section_LM}

In this section, we will prove Proposition~\ref{prop_LM} for $\LM $. 
First, we prove the following lemma.  

\begin{lem}\label{lem_LM}
For $n\geq 1$ and any parity-reversing element $r^\prime \in \LM $, $\LM $ is generated by $h_i$ for $1\leq i\leq 2n$, $t_{1,2}$, and $r^\prime $. 
\end{lem}

\begin{proof}
Let $r^\prime \in \LM $ be a parity-reversing element and $G$ a subgroup of $\LM $ which is generated by $h_i$ for $1\leq i\leq 2n$, $t_{1,2}$, and $r^\prime $. 
Recall that we have the surjective homomorphism $\Psi \colon \LM \to W_{2n+2}$ by Lemma~\ref{lem_GW} and $\Psi (h_i)=(i\ i+2)$ for $1\leq i\leq 2n$. 
Since $S_{n+1}^o\times S_{n+1}^e$ is generated by $(i\ i+2)$ for $1\leq i\leq 2n$ (for precise definition, see before Lemma~\ref{lem_GW}), by the exact sequence~(\ref{exact1}), $W_{2n+2}$ is generated by $\Psi (h_i)$ for $1\leq i\leq 2n$ and $\Psi (r^\prime )$. 
By Proposition~3.1 in \cite{Hirose-Omori}, $\PM $ is generated by $t_{i,j}$ for $1\leq i<j\leq 2n+1$ with $(i,j)\not =(1,2n+1)$. 
Thus, by the exact sequence~(\ref{exact2}), $\LM $ is generated by $h_i$ for $1\leq i\leq 2n$, $t_{i,j}$ for $1\leq i<j\leq 2n+1$ with $(i,j)\not =(1,2n+1)$, and $r^\prime $. 
It is enough for completing the proof of Lemma~\ref{lem_LM} to prove that $t_{i,j}\in G$ for $(i,j)\not =(1,2)$. 

By the relation~(2)~(a) of Theorem~5.3 in~\cite{Hirose-Omori}, we have the relation $h_it_{i,i+1}h_i^{-1}=t_{i+1,i+2}$ for $1\leq i\leq 2n$. 
Thus we have $t_{i,i+1}\in G$ for $2\leq i\leq 2n$. 
By the relation~(4)~(a) of Theorem~5.3 in~\cite{Hirose-Omori}, we have the relations
\[
t_{i,j}=\left\{
		\begin{array}{ll}
		h_i^2 \quad \text{for }j-i=2,\\
		t_{j-1,j}^{-\frac{j-i-3}{2}}\cdots t_{i+2,i+3}^{-\frac{j-i-3}{2}}t_{i,i+1}^{-\frac{j-i-3}{2}}(h_{j-2}\cdots h_{i+1}h_{i})^{\frac{j-i+1}{2}}\\ \text{for }2n-1\geq j-i \geq 3\text{ is odd},\\
		t_{j-2,j-1}^{-\frac{j-i-2}{2}}\cdots t_{i+2,i+3}^{-\frac{j-i-2}{2}}t_{i,i+1}^{-\frac{j-i-2}{2}}h_{j-2}\cdots h_{i+2}h_{i}h_{i}h_{i+2}\cdots h_{j-2}\\ \cdot (h_{j-3}\cdots h_{i+1}h_{i})^{\frac{j-i}{2}}\quad \text{for }2n\geq j-i \geq 4\text{ is even}. 
		\end{array}
		\right.
\]
Therefore, we have $t_{i,j}\in G$ for $j-i\geq 2$ and have completed the proof of Lemma~\ref{lem_LM}. 
\end{proof}

\begin{proof}[Proof of Proposition~\ref{prop_LM} for $\LM $]
Let $G$ be a subgroup of $\LM $ which is generated by $h_1$, $t_{1,2}$, and $r_1$. 
Since $r_1$ is a parity reversing element, by Lemma~\ref{lem_LM}, $\LM $ is generated by $h_i$ for $1\leq i\leq 2n$, $t_{1,2}$, and $r_1$. 
Hence it is enough for completing the proof of Proposition~\ref{prop_LM} to prove that $h_{i}\in G$ for $2\leq i\leq 2n$. 

By the definition of $r_1$, we have $r_1(l_i\cup l_{i+1})=l_{i+1}\cup l_{i+2}$ for $1\leq i\leq 2n-1$. 
Thus we have $r_1h_{i}r_1^{-1}=h_{i+1}$ for $1\leq i\leq 2n-1$ and $h_{i}\in G$ for $2\leq i\leq 2n$. 
Therefore we have completed the proof of Proposition~\ref{prop_LM}. 
\end{proof}

\subsection{Proofs of the main theorem for $\boldsymbol{\LMp }$ and $\boldsymbol{\LMb }$}\label{section_LMp}

In this section, we will prove Proposition~\ref{prop_LM} for $\LMp $ and $\LMb $. 
By Corollary~5.4 in~\cite{Hirose-Omori}, we have the following lemma. 

\begin{lem}[Corollary~5.4 in~\cite{Hirose-Omori}]\label{lem_LMp}
For $n\geq 1$, $\LMp$ and $\LMb $ are generated by $h_1,\ h_2,\ \dots ,\ h_{2n-1}$, and $t_{1,2}$. 
\end{lem}

To prove Proposition~\ref{prop_LM} for $\LMp $ and $\LMb $, we prepare the following lemma. 

\begin{lem}\label{lem_rel_LMp}
For $n\geq 2$ and $1\leq i\leq 2n-3$, the relation 
\[
(h_{2n-1}\cdots h_2h_1t_{1,2})^{-1}h_i(h_{2n-1}\cdots h_2h_1t_{1,2})=h_{i+2}
\]
holds in $\LMp$ and $\LMb $. 
\end{lem}

\begin{proof}
First, since we can check that $h_{i}^{-1}h_{i+1}^{-1}h_{i+2}^{-1}(l_i\cup l_{i+1})=l_{i+2}\cup l_{i+3}$ for $1\leq i\leq 2n-3$, we have $h_{i}^{-1}h_{i+1}^{-1}h_{i+2}^{-1}h_ih_{i+2}h_{i+1}h_{i}=h_{i+2}$ for $1\leq i\leq 2n-3$ (that is equivalent to the relation~(2)~(c) of Theorem~5.1 in \cite{Hirose-Omori}). 
Thus we have 
\begin{eqnarray*}
&&t_{1,2}^{-1}h_1^{-1}h_2^{-1}\cdots h_{i+2}^{-1}h_{i+3}^{-1}\cdots h_{2n-1}^{-1}\cdot h_i\cdot \underset{\leftarrow }{\underline{h_{2n-1}\cdots h_{i+3}}}h_{i+2}\cdots h_2h_1t_{1,2}\\
&\overset{\text{COMM}}{\underset{}{=}}&t_{1,2}^{-1}h_1^{-1}h_2^{-1}\cdots h_{i-1}^{-1}\underline{h_{i}^{-1}h_{i+1}^{-1}h_{i+2}^{-1}\cdot h_i\cdot h_{i+2}h_{i+1}h_{i}}h_{i-1}\cdots h_2h_1t_{1,2}\\
&\overset{}{\underset{}{=}}&t_{1,2}^{-1}h_1^{-1}h_2^{-1}\cdots h_{i-1}^{-1}\underset{\leftarrow }{\underline{h_{i+2}}}h_{i-1}\cdots h_2h_1t_{1,2}\\
&\overset{\text{COMM}}{\underset{}{=}}&h_{i+2}. 
\end{eqnarray*}
Therefore, we have completed the proof of Lemma~\ref{lem_rel_LMp}. 
\end{proof}

\begin{proof}[Proof of Proposition~\ref{prop_LM} for $\LMp $ and $\LMb $]
Since Theorem~\ref{thm_lmod} for $\LMp $ and $\LMb $ when $n=1$ is obtained from Lemma~\ref{lem_LMp}, we assume that $n\geq 2$. 
We will prove Proposition~\ref{prop_LM} for $\LMp $ (for the case of $\LMb $, we can prove by a similar argument below). 
Let $G$ be a subgroup of $\LMp $ which is generated by $h_1$, $h_2$, and $h_{2n-1}\cdots h_2h_1t_{1,2}$. 
By Lemma~\ref{lem_LMp}, $\LMp $ is generated by $h_i$ for $1\leq i\leq 2n-1$ and $t_{1,2}$. 
Hence it is enough for completing the proof of Proposition~\ref{prop_LM} for $\LMp $ to prove that $h_{i}\in G$ for $3\leq i\leq 2n$ and $t_{1,2}\in G$. 

Since the relation $(h_{2n-1}\cdots h_2h_1t_{1,2})^{-1}h_i(h_{2n-1}\cdots h_2h_1t_{1,2})=h_{i+2}$ holds in $\LMp $ for $1\leq i\leq 2n-3$ by Lemma~\ref{lem_rel_LMp}, we have $h_i\in G$ for $3\leq i\leq 2n-1$. 
Hence $h_i$ lies in $G$ for any $1\leq i\leq 2n-1$, we have 
\[
t_{1,2}=h_1^{-1}h_2^{-1}\cdots h_{2n-1}^{-1}\cdot (h_{2n-1}\cdots h_2h_1t_{1,2})\in G.
\] 
Therefore we have completed the proof of Proposition~\ref{prop_LM}. 
\end{proof}

\section{Small generating sets for balanced superelliptic mapping class groups}\label{section_smod}

%In this section, we will prove Theorem~\ref{thm_smod}. 
Throughout this section, we assume that $g=n(k-1)$ for $n\geq 1$ and $k\geq 3$.  
From Section~\ref{section_lifts}, we recall that 
\begin{itemize}
\item $\widetilde{t}_{i,i+1}=t_{\gamma _i^1}t_{\gamma _i^2}\cdots t_{\gamma _i^{k}}$ for $1\leq i\leq 2n+1$, 
\item $\widetilde{h}_i=t_{\gamma _i^1}t_{\gamma _{i+1}^1}t_{\gamma _i^2}t_{\gamma _{i+1}^2}\cdots t_{\gamma _i^{k-1}}t_{\gamma _{i+1}^{k-1}}t_{\gamma _i^{k}}$ for odd $1\leq i\leq 2n-1$,  
\item $\widetilde{h}_i=t_{\gamma _i^{k}}t_{\gamma _{i+1}^{k}}t_{\gamma _i^{k-1}}t_{\gamma _{i+1}^{k-1}}\cdots t_{\gamma _i^{2}}t_{\gamma _{i+1}^{2}}t_{\gamma _i^{1}}$ for even $2\leq i\leq 2n$ (see Figure~\ref{fig_scc_c_il}),
\item $\widetilde{r}_1=\widetilde{r}(\widetilde{h}_1^{-1}\widetilde{h}_2^{-1}\cdots \widetilde{h}_{2n-2}^{-1})(\widetilde{h}_1^{-1}\widetilde{h}_2^{-1}\cdots \widetilde{h}_{2n-4}^{-1})\cdots (\widetilde{h}_1^{-1}\widetilde{h}_2^{-1})\widetilde{h}_{2n-1}^{-1}\cdots \widetilde{h}_3^{-1}\widetilde{h}_1^{-1}\\ \widetilde{t}_{2n,2n+1}^{n-1}\cdots \widetilde{t}_{6,7}^2\widetilde{t}_{4,5}$ for $n\geq 2$, 
\item $\widetilde{r}_1=\widetilde{r}\widetilde{h}_1^{-1}$ for $n=1$. 
\end{itemize}
%Theorem~\ref{thm_smod} follows from the next proposition. 

In this section, we will prove Theorem~\ref{thm_smod} by showing the following proposition. % Proposition~\ref{prop_LM}. 
\begin{prop}\label{prop_SM}
For $n\geq 2$ and $k\geq 3$ with $g=n(k-1)$, 
\begin{enumerate}
\item $\SM $ is generated by $\widetilde{h}_1$, $\widetilde{t}_{1,2}$, and $\widetilde{r}_1$, 
\item $\SMp $ and $\SMb $ are generated by $\widetilde{h}_1$, $\widetilde{h}_2$, and $\widetilde{h}_{2n-1}\cdots \widetilde{h}_2\widetilde{h}_1\widetilde{t}_{1,2}$ when $n\geq 2$, 
\item $\SMp $ and $\SMb $ are generated by $\widetilde{h}_1$ and $\widetilde{t}_{1,2}$. 
\end{enumerate}
\end{prop} 

Let $\theta \colon \SM \to \LM $ be the surjective homomorphism with the kernel $\left< \zeta \right> $ which is obtained from the Birman-Hilden correspondence~\cite{Birman-Hilden2}, namely $\theta $ is defined as follows: for $f\in \SM $ and a symmetric representative $\varphi \in f$ for $\zeta $, we define $\hat{\varphi }\colon \Sigma _0\to \Sigma _0$ by $\hat{\varphi }(x)=p(\varphi (\widetilde{x}))$ for some $\widetilde{x}\in p^{-1}(x)$. 
Then $\theta (f)=[\hat{\varphi }]\in \LM $. 
As a similar construction, we have an isomorphism $\theta ^1\colon \SMb \to \LMb $ by Birman and Hilden~\cite{Birman-Hilden2}.

\begin{proof}[Proof of Proposition~\ref{prop_SM} for $\SMp $ and $\SMb $]
When $n=1$, Proposition~\ref{prop_SM} for $\SMp $ and $\SMb $ is obtained from Corollary~6.13 in \cite{Hirose-Omori}. 
Assume that $n\geq 2$. 
Since $\widetilde{t}_{i,i+1}$ for $1\leq i\leq 2n+1$ and $\widetilde{h}_i$ for $1\leq i\leq 2n-1$ are lifts of $t_{i,i+1}$ and $h_i$ with respect to $p$ by Lemmas~\ref{lift-t_{i,i+1}} and \ref{lift-h_i}, respectively, Proposition~\ref{prop_SM} for $\SMb $ is immediately obtained from Proposition~\ref{prop_LM} for $\LMb $ and the isomorphism~$\theta ^1$. 
The inclusion $\widetilde{\iota }\colon \Sigma _g^1\hookrightarrow \Sigma _g=\Sigma _g^1\cup \widetilde{D}$ induces the surjective homomorphism $\widetilde{\iota }_\ast \colon \Mgb \to \mathrm{Mod}_{g,1}$ by extending homeomorphisms by the identity map on $\widetilde{D}$. 
The homomorphism $\widetilde{\iota }_\ast$ is called the \textit{capping homomorphism}. 
By Lemma~4.9 in~\cite{Hirose-Omori}, we have $\widetilde{\iota }_\ast (\SMb )=\SMp $. 
Since we have $\widetilde{\iota }_\ast (\widetilde{t}_{i,i+1})=\widetilde{t}_{i,i+1}$ for $1\leq i\leq 2n+1$ and $\widetilde{\iota }_\ast (\widetilde{h}_i)=\widetilde{h}_i$ for $1\leq i\leq 2n-1$ by the definitions, %by Proposition~\ref{prop_LMp}, 
$\SMp $ is also generated by $\widetilde{h}_1$, $\widetilde{h}_2$, and $\widetilde{h}_{2n-1}\cdots \widetilde{h}_2\widetilde{h}_1\widetilde{t}_{1,2}$. 
Therefore, we have completed the proof of Proposition~\ref{prop_SM} for $\SMp $ and $\SMb $. 
\end{proof}

To prove Proposition~\ref{prop_SM} for $\SM $, we prepare the following lemmas. 

\begin{lem}\label{lem_conj_r_1_t}
The relation
\[
\widetilde{r}_1\widetilde{t}_{i,i+1}\widetilde{r}_1^{-1}=\widetilde{t}_{i+1,i+2}
\]
for $1\leq i\leq 2n$ holds in $\SM $. 
\end{lem}

\begin{proof}
%First, we will prove the case~(2). 
By the definition of $\widetilde{t}_{i,i+1}$ and conjugation relations, we have 
\[
\widetilde{r}_1\widetilde{t}_{i,i+1}\widetilde{r}_1^{-1}=\widetilde{r}_1t_{\gamma _i^1}t_{\gamma _i^2}\cdots t_{\gamma _i^{k}}\widetilde{r}_1^{-1}=t_{\widetilde{r}_1(\gamma _i^1)}t_{\widetilde{r}_1(\gamma _i^2)}\cdots t_{\widetilde{r}_1(\gamma _i^{k})}
\]
for $1\leq i\leq 2n$. 
Since $\widetilde{r}_1$ is a lift of $r_1$ by Lemma~\ref{lift-r_1} and $\gamma _i^l$ for $1\leq l\leq k$ is a lift of $\gamma _{i,i+1}$ with respect to $p$, the simple closed curve $\widetilde{r}_1(\gamma _i^{l})$ for $1\leq l\leq k$ is also a lift of $r_1(\gamma _{i,i+1})=\gamma _{i+1,i+2}$ with respect to $p$. 
Hence we have $\{ \widetilde{r}_1(\gamma _i^{1}),\ \widetilde{r}_1(\gamma _i^{2}), \dots ,\ \widetilde{r}_1(\gamma _i^{k})\} =\{ \gamma _{i+1}^{1},\ \gamma _{i+1}^{2}, \dots ,\ \gamma _{i+1}^{k}\} $ and 
\[
t_{\widetilde{r}_1(\gamma _i^1)}t_{\widetilde{r}_1(\gamma _i^2)}\cdots t_{\widetilde{r}_1(\gamma _i^{k})}=t_{\gamma _{i+1}^1}t_{\gamma _{i+1}^2}\cdots t_{\gamma _{i+1}^{k}}=\widetilde{t}_{i+1,i+2}
\]
in $\SM $. 
Therefore, we have completed the proof of Lemma~\ref{lem_conj_r_1_t}.  
\end{proof}

\begin{lem}\label{lem_conj_r_1_h}
The relation
\[
\widetilde{r}_1\widetilde{h}_i\widetilde{r}_1^{-1}=\widetilde{h}_{i+1}
\]
for $1\leq i\leq 2n-1$ holds in $\SM $. 
\end{lem}

\begin{proof}
When $n=1$, we have 
\begin{eqnarray*}
\widetilde{r}_1(\gamma _1^l)&=&\widetilde{r}\widetilde{h}_1^{-1}(\gamma _1^l)=\widetilde{r}(t_{\gamma _1^{k}}^{-1}t_{\gamma _{2}^{k-1}}^{-1}t_{\gamma _1^{k-1}}^{-1}\cdots t_{\gamma _{2}^2}^{-1}t_{\gamma _1^2}^{-1}t_{\gamma _{2}^1}^{-1}t_{\gamma _1^1}^{-1}(\gamma _1^l))\\
&=&\widetilde{r}(\gamma _2^{l-1})=\gamma _2^{k-l+1}
\end{eqnarray*}
for $2\leq l\leq k$. 
Similarly, when $n=1$, we have $\widetilde{r}_1(\gamma _1^1)=\gamma _2^k$ and $\widetilde{r}_1(\gamma _2^l)=\gamma _3^{k-l-1}$ for $1\leq l\leq k$. 
Thus, we have 
\begin{eqnarray*}
\widetilde{r}_1\widetilde{h}_1\widetilde{r}_1^{-1}&=&\widetilde{r}_1t_{\gamma _1^1}t_{\gamma _{2}^1}t_{\gamma _1^2}t_{\gamma _{2}^2}\cdots t_{\gamma _1^{k-1}}t_{\gamma _{2}^{k-1}}t_{\gamma _1^{k}}\widetilde{r}_1^{-1}\\
&=&t_{\widetilde{r}_1(\gamma _1^1)}t_{\widetilde{r}_1(\gamma _{2}^1)}t_{\widetilde{r}_1(\gamma _1^2)}t_{\widetilde{r}_1(\gamma _{2}^2)}\cdots t_{\widetilde{r}_1(\gamma _1^{k-1})}t_{\widetilde{r}_1(\gamma _{2}^{k-1})}t_{\widetilde{r}_1(\gamma _1^{k})}\\
&=&t_{\gamma _{2}^{k}}t_{\gamma _{3}^{k}}t_{\gamma _{2}^{k-1}}t_{\gamma _{3}^{k-1}}\cdots t_{\gamma _{2}^{2}}t_{\gamma _{3}^{2}}t_{\gamma _{2}^{1}}\\
&=&\widetilde{h}_2.
\end{eqnarray*}
Thus Lemma~\ref{lem_conj_r_1_h} holds for $n=1$. 

Assume that $n\geq 2$ and $1\leq i\leq 2n+1$ is odd.  
Then, we have $\widetilde{h}_i=t_{\gamma _i^1}t_{\gamma _{i+1}^1}t_{\gamma _i^2}t_{\gamma _{i+1}^2}\cdots t_{\gamma _i^{k-1}}t_{\gamma _{i+1}^{k-1}}t_{\gamma _i^{k}}$. 
A collection $(c_1,\ c_2,\ \dots ,\ c_m)$ of simple closed curves $c_1,\ c_2,\ \dots ,\ c_m$ on $\Sigma _g$ is a $m$\textit{-chain} on $\Sigma _g$ if $c_i$ transversely intersects with $c_{i+1}$ at one point for $1\leq i\leq m-1$ and $c_i$ is disjoint from $c_j$ for $j\not \in \{ i-1,\ i,\ i+1 \}$. 
Since $\widetilde{r}_1$ is a homeomorphism and $\bigl( \gamma _i^1,\ \gamma _{i+1}^1,\ \gamma _i^2,\ \gamma _{i+1}^2,\ \cdots \gamma _i^{k-1},\ \gamma _{i+1}^{k-1},\ \gamma _i^{k}\bigr) $ is a $(2k-1)$-chain on $\Sigma _g$, the collection $\bigl( \widetilde{r}_1(\gamma _i^1),\ \widetilde{r}_1(\gamma _{i+1}^1),\ \widetilde{r}_1(\gamma _i^2),\ \widetilde{r}_1(\gamma _{i+1}^2),\ \cdots \widetilde{r}_1(\gamma _i^{k-1}),\ \widetilde{r}_1(\gamma _{i+1}^{k-1}),\ \widetilde{r}_1(\gamma _i^{k})\bigr) $ is also a $(2k-1)$-chain on $\Sigma _g$. 
Put $\mathcal{C}_{i}=\{ \gamma _{i}^{1},\ \gamma _{i}^{2}, \dots ,\ \gamma _{i}^{k}\}$ for $1\leq i\leq 2n+1$. 
Since $\{ \widetilde{r}_1(\gamma _i^{1}),\ \widetilde{r}_1(\gamma _i^{2}), \dots ,\ \widetilde{r}_1(\gamma _i^{k})\} =\mathcal{C}_{i+1}$ for $1\leq i\leq 2n$ by an argument in the proof of Lemma~\ref{lem_conj_r_1_t}, we have $\bigl( \widetilde{r}_1(\gamma _i^1),\ \widetilde{r}_1(\gamma _{i+1}^1),\ \widetilde{r}_1(\gamma _i^2),\ \widetilde{r}_1(\gamma _{i+1}^2),\ \cdots \widetilde{r}_1(\gamma _i^{k-1}),\ \widetilde{r}_1(\gamma _{i+1}^{k-1}),\ \widetilde{r}_1(\gamma _i^{k})\bigr) \in \mathcal{C}_{i+1}\times \mathcal{C}_{i+2}\times \mathcal{C}_{i+1}\times \mathcal{C}_{i+2}\times \cdots \times \mathcal{C}_{i+1}\times \mathcal{C}_{i+2}\times \mathcal{C}_{i+1}$. 
%Put $\bar{\gamma }_{i+1}^l=\widetilde{r}_1(\gamma _i^{l})$ for . 

Since $\widetilde{h}_i$ and $\widetilde{r}_1$ have symmetric representative for $\zeta$, the product 
\[
\widetilde{r}_1\widetilde{h}_i\widetilde{r}_1^{-1}=t_{\widetilde{r}_1(\gamma _i^1)}t_{\widetilde{r}_1(\gamma _{i+1}^1)}t_{\widetilde{r}_1(\gamma _i^2)}t_{\widetilde{r}_1(\gamma _{i+1}^2)}\cdots t_{\widetilde{r}_1(\gamma _i^{k-1})}t_{\widetilde{r}_1(\gamma _{i+1}^{k-1})}t_{\widetilde{r}_1(\gamma _i^{k})}
\]
also has a symmetric representative for $\zeta$, namely, there exists an integer $1\leq l^\prime \leq k-1$ such that $(\widetilde{r}_1\widetilde{h}_i\widetilde{r}_1^{-1})\zeta (\widetilde{r}_1\widetilde{h}_i\widetilde{r}_1^{-1})^{-1}=\zeta ^{l^\prime }$ in $\SM $. 
Since $n\geq 2$, for $1\leq l\leq k$, we have $\widetilde{r}_1\widetilde{h}_i\widetilde{r}_1^{-1}(\gamma _1^l)=\gamma _1^l$ or $\widetilde{r}_1\widetilde{h}_i\widetilde{r}_1^{-1}(\gamma _{2n+1}^l)=\gamma _{2n+1}^l$ with orientations. 
Hence, for $1\leq l\leq k$, we also have $\zeta ^{l^\prime }(\gamma _1^l)=(\widetilde{r}_1\widetilde{h}_i\widetilde{r}_1^{-1})\zeta (\widetilde{r}_1\widetilde{h}_i\widetilde{r}_1^{-1})^{-1}(\gamma _1^l)=\zeta (\gamma _1^l)$ or $\zeta ^{l^\prime }(\gamma _{2n+1}^l)=(\widetilde{r}_1\widetilde{h}_i\widetilde{r}_1^{-1})\zeta (\widetilde{r}_1\widetilde{h}_i\widetilde{r}_1^{-1})^{-1}(\gamma _{2n+1}^l)=\zeta (\gamma _{2n+1}^l)$ with orientations.   
Thus, $l^\prime =1$ and $\widetilde{r}_1\widetilde{h}_i\widetilde{r}_1^{-1}$ commutes with $\zeta $. 

If $\widetilde{r}_1\widetilde{h}_i\widetilde{r}_1^{-1}=t_{\gamma _{i+1}^{l^\prime }}t_{\gamma _{i+2}^{l^\prime +1}}t_{\gamma _{i+1}^{l^\prime +1}}\cdots t_{\gamma _{i+2}^{l^\prime +k-2}}t_{\gamma _{i+1}^{l^\prime +k-2}}t_{\gamma _{i+2}^{l^\prime +k-1}}t_{\gamma _{i+1}^{l^\prime +k-1}}$ for $1\leq i \leq 2n-1$ and some $1\leq l^\prime \leq k$, where we consider the indices $l^\prime , l^\prime +1, \dots , l^\prime +k-1$ mod $k$, then $\widetilde{r}_1\widetilde{h}_i\widetilde{r}_1^{-1}$ does not commute with $\zeta $. 
In fact, it is necessarily for the coincidence of $\zeta t_{\gamma _{i+1}^{l^\prime }}t_{\gamma _{i+2}^{l^\prime +1}}t_{\gamma _{i+1}^{l^\prime +1}}\cdots t_{\gamma _{i+2}^{l^\prime +k-2}}t_{\gamma _{i+1}^{l^\prime +k-2}}t_{\gamma _{i+2}^{l^\prime +k-1}}t_{\gamma _{i+1}^{l^\prime +k-1}}\zeta ^{-1}=t_{\gamma _{i+1}^{l^\prime +1}}t_{\gamma _{i+2}^{l^\prime +2}}t_{\gamma _{i+1}^{l^\prime +2}}\cdots t_{\gamma _{i+2}^{l^\prime +k-1}}t_{\gamma _{i+1}^{l^\prime +k-1}}t_{\gamma _{i+2}^{l^\prime +k}}t_{\gamma _{i+1}^{l^\prime +k}}$ and $t_{\gamma _{i+1}^{l^\prime }}t_{\gamma _{i+2}^{l^\prime +1}}t_{\gamma _{i+1}^{l^\prime +1}}\cdots t_{\gamma _{i+2}^{l^\prime +k-2}}t_{\gamma _{i+1}^{l^\prime +k-2}}t_{\gamma _{i+2}^{l^\prime +k-1}}t_{\gamma _{i+1}^{l^\prime +k-1}}$ to hold the relation
\begin{align*}
&(t_{\gamma _{i+1}^{l^\prime +1}}t_{\gamma _{i+2}^{l^\prime +2}}t_{\gamma _{i+1}^{l^\prime +2}}\cdots t_{\gamma _{i+2}^{l^\prime +k-1}}t_{\gamma _{i+1}^{l^\prime +k-1}})t_{\gamma _{i+2}^{l^\prime +k}}t_{\gamma _{i+1}^{l^\prime +k}}(t_{\gamma _{i+1}^{l^\prime +1}}t_{\gamma _{i+2}^{l^\prime +2}}t_{\gamma _{i+1}^{l^\prime +2}}\cdots t_{\gamma _{i+2}^{l^\prime +k-1}}t_{\gamma _{i+1}^{l^\prime +k-1}})^{-1}\\
&=t_{\gamma _{i+1}^{l^\prime }}t_{\gamma _{i+2}^{l^\prime +1}}
\end{align*}
in $\SM $. 
However, this relation does not hold in $\SM $. 
Thus, we have $\widetilde{r}_1\widetilde{h}_i\widetilde{r}_1^{-1}=t_{\gamma _{i+1}^{l^\prime +k-1}}t_{\gamma _{i+2}^{l^\prime +k-1}}t_{\gamma _{i+1}^{l^\prime +k-2}}t_{\gamma _{i+2}^{l^\prime +k-2}}\cdots t_{\gamma _{i+1}^{l^\prime +1}}t_{\gamma _{i+2}^{l^\prime +1}}t_{\gamma _{i+1}^{l^\prime }}$ for some $1\leq l^\prime \leq k$, where we consider the indices $l^\prime , l^\prime +1, \dots , l^\prime +k-1$ mod $k$. 
Since we have the relation 
\begin{align*}
&(t_{\gamma _{i+1}^{l^\prime +k-2}}t_{\gamma _{i+2}^{l^\prime +k-2}}\cdots t_{\gamma _{i+1}^{l^\prime +1}}t_{\gamma _{i+2}^{l^\prime +1}}t_{\gamma _{i+1}^{l^\prime }})^{-1}t_{\gamma _{i+1}^{l^\prime +k-1}}t_{\gamma _{i+2}^{l^\prime +k-1}}(t_{\gamma _{i+1}^{l^\prime +k-2}}t_{\gamma _{i+2}^{l^\prime +k-2}}\cdots t_{\gamma _{i+1}^{l^\prime +1}}t_{\gamma _{i+2}^{l^\prime +1}}t_{\gamma _{i+1}^{l^\prime }})\\
&=t_{\gamma _{i+2}^{l^\prime }}t_{\gamma _{i+1}^{l^\prime -1}}
\end{align*}
in $\SM$, by using this relation inductively, we have
\begin{eqnarray*}
\widetilde{r}_1\widetilde{h}_i\widetilde{r}_1^{-1}&=&t_{\gamma _{i+1}^{l^\prime +k-1}}t_{\gamma _{i+2}^{l^\prime +k-1}}t_{\gamma _{i+1}^{l^\prime +k-2}}t_{\gamma _{i+2}^{l^\prime +k-2}}\cdots t_{\gamma _{i+1}^{l^\prime +1}}t_{\gamma _{i+2}^{l^\prime +1}}t_{\gamma _{i+1}^{l^\prime }}\\
&=&t_{\gamma _{i+1}^{k}}t_{\gamma _{i+2}^{k}}t_{\gamma _{i+1}^{k-1}}t_{\gamma _{i+2}^{k-1}}\cdots t_{\gamma _{i+1}^{2}}t_{\gamma _{i+2}^{2}}t_{\gamma _{i+1}^{1}}\\
&=&\widetilde{h}_{i+1}. 
\end{eqnarray*}
Thus the relation in Lemma~\ref{lem_conj_r_1_h} for odd $1\leq i\leq 2n-1$ holds in $\SM $.  
By an argument similar to the odd $i$ and $n\geq 2$ case, we can show that the relation in Lemma~\ref{lem_conj_r_1_h} for even $2\leq i\leq 2n-2$ and $n\geq 2$ also holds in $\SM $. 
Therefore we have completed the proof of Lemma~\ref{lem_conj_r_1_h}. 
\end{proof}

By Lemma~4.6 in \cite{Hirose-Omori}, the forgetful homomorphism $\mathcal{F}\colon \SMp \to \SM $ which is defined by the correspondence $[\varphi ]\mapsto [\varphi ]$ is injective. 
Hence we regard $\SMp $ as a subgroup of $\SM $. 
Recall that we have the surjective homomorphism $\widetilde{\iota }_\ast \colon \SMb \to \SMp $ defined in the proof of Proposition~\ref{prop_SM} for $\SMp $ and $\SMb $.

\begin{proof}[Proof of Proposition~\ref{prop_SM} for $\SM $]
%Let $G$ be a subgroup of $\SM $ which is generated by $\widetilde{h}_1$, $\widetilde{t}_{1,2}$, and $\widetilde{r}_1$. 
From the Birman-Hilden correspondence~\cite{Birman-Hilden2}, we have the following exact sequence:  
\begin{eqnarray*}
1\longrightarrow \left< \zeta \right>  \longrightarrow \SM \stackrel{\theta }{\longrightarrow }\LM \longrightarrow 1. 
\end{eqnarray*}
Since $\LM $ is generated by $h_1$, $t_{1,2}$, and $r_1$ for $n\geq 1$ by Proposition~\ref{prop_LM} and $\widetilde{h}_1$, $\widetilde{t}_{1,2}$, and $\widetilde{r}_1$ are lifts of $h_1$, $t_{1,2}$, and $r_1$ with respect to $p$ by Lemmas~\ref{lift-t_{i,i+1}}, \ref{lift-h_i}, and \ref{lift-r_1}, respectively, $\SM $ is generated by $\widetilde{h}_1$, $\widetilde{t}_{1,2}$, $\widetilde{r}_1$, and $\zeta $ by the exact sequence above. 
It is enough for completing the proof of Proposition~\ref{prop_SM} to prove that $\zeta $ is a product of $\widetilde{h}_1$, $\widetilde{t}_{1,2}$, and $\widetilde{r}_1$. 

Let $\zeta ^\prime $ be a self-homeomorphism on $\Sigma _g$ which is described as a result of a $(-\frac{2\pi }{k})$-rotation of $\Sigma _g^1\subset \Sigma _g$ fixing the disk $\widetilde{D}$ pointwise (for instance, in the case of $k=3$, $\zeta ^\prime $ is a $(-\frac{2\pi }{3})$-rotation as on the top in Figure~\ref{fig_lift_t_partial-d}). 
Since the restriction of $\zeta ^\prime $ to $\widetilde{D}$ is identity, we regard $\zeta ^\prime $ as an element in $\SMb $. 
The isotopy class of a self-homeomorphism on $\Sigma _0$ (resp. $\Sigma _g$) relative to $D\cup \B $ (resp. $\widetilde{D}$) is determined by the isotopy class of the image of $L$ (resp. $\widetilde{L}=p^{-1}(L)$) relative to $D\cup \B $ (resp. $\widetilde{D}$). 
Since we can show that $p(\zeta ^\prime (\widetilde{L}))$ is isotopic to $t_{\partial D}(L)$ relative to $D\cup \B $ (in the case $k=3$, see Figure~\ref{fig_lift_t_partial-d}), $\zeta ^\prime $ is a lift of $t_{\partial D}=t_{1,2n+1}$ with respect to $p$. 

By the relations~(4)~(a) of Theorems~5.1 and 6.10 in~\cite{Hirose-Omori}, we have  
\[
t_{1,2n+1}=t_{2n-1,2n}^{-n+1}\cdots t_{3,4}^{-n+1}t_{1,2}^{-n+1}h_{2n-1}\cdots h_3h_1h_1h_3\cdots h_{2n-1}(h_{2n-2}\cdots h_{2}h_{1})^{n}
\]
in $\LMb $ and its lift 
\[
\zeta ^\prime =\widetilde{t}_{2n-1,2n}^{-n+1}\cdots \widetilde{t}_{3,4}^{-n+1}\widetilde{t}_{1,2}^{-n+1}\widetilde{h}_{2n-1}\cdots \widetilde{h}_3\widetilde{h}_1\widetilde{h}_1\widetilde{h}_3\cdots \widetilde{h}_{2n-1}(\widetilde{h}_{2n-2}\cdots \widetilde{h}_{2}\widetilde{h}_{1})^{n}
\]
in $\SMb $. 
By the definition of $\zeta ^\prime $, we have $\widetilde{\iota }_\ast (\zeta ^\prime )=\zeta $ clearly. 
Thus we have 
\[
\zeta =\widetilde{t}_{2n-1,2n}^{-n+1}\cdots \widetilde{t}_{3,4}^{-n+1}\widetilde{t}_{1,2}^{-n+1}\widetilde{h}_{2n-1}\cdots \widetilde{h}_3\widetilde{h}_1\widetilde{h}_1\widetilde{h}_3\cdots \widetilde{h}_{2n-1}(\widetilde{h}_{2n-2}\cdots \widetilde{h}_{2}\widetilde{h}_{1})^{n}
\]
in $\SM $. 
By Lemmas~\ref{lem_conj_r_1_t} and \ref{lem_conj_r_1_t}, we have the relations $\widetilde{h}_{i}=\widetilde{r}_1^{i-1}\widetilde{h}_1\widetilde{r}_1^{-(i-1)}$ for $2\leq i\leq 2n-1$ and $\widetilde{t}_{i,i+1}=\widetilde{r}_1^{i-1}\widetilde{t}_{1,2}\widetilde{r}_1^{-(i-1)}$ for $2\leq i\leq 2n-1$. 
Therefore $\zeta $ is a product of $\widetilde{h}_1$, $\widetilde{t}_{1,2}$, and $\widetilde{r}_1$ and we have completed the proof of Proposition~\ref{prop_SM}.  
\end{proof}

\begin{figure}[h]
\includegraphics[scale=1.3]{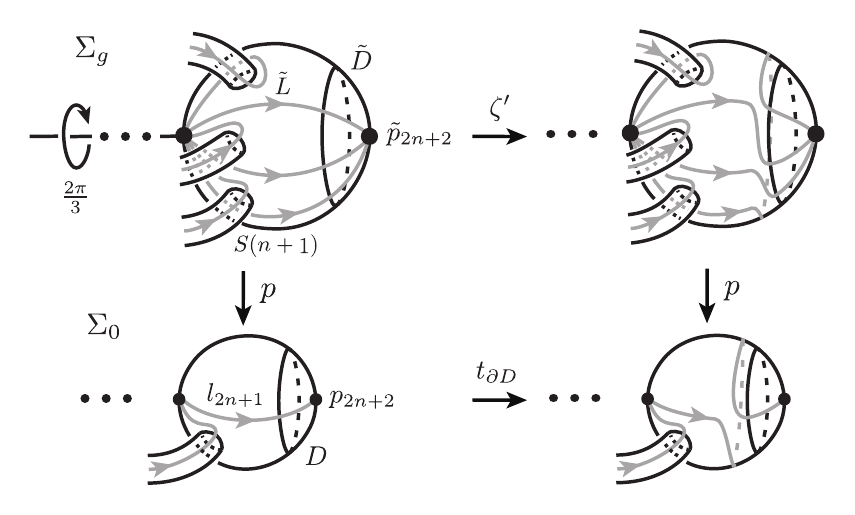}
\caption{The images of $L$ by $t_{\partial D}$ and of $\widetilde{L}$ by $\zeta ^\prime $ for $k=3$.}\label{fig_lift_t_partial-d}
\end{figure}

\par
{\bf Acknowledgement:} The author would like to express his gratitude to Susumu Hirose, for his encouragement and helpful advices. 
The author was supported by JSPS KAKENHI Grant Numbers JP19K23409 and 21K13794.
%The authors also wish to thank Susumu Hirose for his comments and helpful advices.
%JST CREST Grant Number JPMJCR17J4, Japan. 
%Grant-in-Aid for Research Activity Start-up

\end{document}